\documentclass[11pt]{amsart}
\usepackage{amssymb,amscd,amsthm,verbatim}
\usepackage{amsbsy}
\usepackage{latexsym}
\usepackage[all,cmtip]{xy}
\usepackage[mathscr]{euscript}

\usepackage{etex}
\usepackage{graphicx}
\usepackage{pictex}
\usepackage{epstopdf}
\usepackage{color}
\usepackage{cite}
\usepackage{nohyperref}
\usepackage[pageref]{backref}
\renewcommand*{\backref}[1]{}
\renewcommand*{\backrefalt}[4]{%
    \ifcase #1 (Not cited.)%
    \or        (Cited on page~#2.)%
    \else      (Cited on pages~#2.)%
    \fi}
\usepackage{bbm}

\date{\today}

\newcommand{\Z}{{\mathbb Z}}
\newcommand{\R}{{\mathbb R}}

\newcommand{\C}{{\mathbb C}}

\newcommand{\T}{{\mathbb T}}

\newcommand{\N}{{\mathbb N}}

\newcommand{\SL}{{\mathrm{SL}}}

\newtheorem{theorem}{Theorem}[section]
\newtheorem{lemma}{Lemma}[section]
\newtheorem{prop}{Proposition}[section]
\newtheorem{coro}{Corollary}[section]

\newtheorem{prob}{Problem}[section]

\theoremstyle{definition}

\newtheorem*{definition}{Definition}

\newtheorem{remark}{Remark}[section]

\theoremstyle{plain}

\allowdisplaybreaks \numberwithin{equation}{section}

\begin{document}

\title[Must the Spectrum of a Random Operator Contain an Interval?]{Must the Spectrum of a Random Schr\"odinger Operator Contain an Interval?}

\author[D.\ Damanik]{David Damanik}

\address{Department of Mathematics, Rice University, Houston, TX~77005, USA}

\email{damanik@rice.edu}

\thanks{D.\ D.\ was supported in part by Simons Fellowship $\# 669836$ and NSF grants DMS--1700131 and DMS--2054752}

\author[A.\ Gorodetski]{Anton Gorodetski}

\address{Department of Mathematics, University of California, Irvine, CA~92697, USA\\
and  National Research University Higher School of Economics, Russian Federation}

\email{asgor@uci.edu}

\thanks{A.\ G.\ was supported in part by NSF grant DMS--1855541 and by Laboratory of Dynamical Systems and Applications NRU HSE, grant of the Ministry of science and higher education of the RF ag. N 075-15-2019-1931.}


\begin{abstract}
We consider Schr\"odinger operators in $\ell^2(\Z)$ whose potentials are given by independent (not necessarily identically distributed) random variables. %We point out that even in the absence of ergodicity, there is a non-random essential spectrum that is common to almost all realizations, and we propose the study of this set and in particular of its topological structure.
%Our very basic goal is to show
We ask whether it is true that almost surely its spectrum %has non-empty interior, that is, it
contains an interval. We provide an affirmative answer %prove this statement for
in the case of random potentials given by a sum of a perturbatively small quasi-periodic potential with analytic sampling function and Diophantine frequency vector and a term of Anderson type, given by independent identically distributed random variables. The proof proceeds by extending a result about the presence of ground states for atypical realizations of the classical Anderson model, which we prove here as well and which appears to be new.
\end{abstract}

\maketitle

\section{Introduction}

In this paper we are interested in the spectrum of random Schr\"odinger operators. Questions of this kind can be asked in any dimension, as well as for both the continuum and the discrete setting. As we explain below, the specific questions we will study turn out to be difficult even in the simplest setting, and hence we choose said setting to carry out our investigation.

Thus, we consider discrete one-dimensional Schr\"odinger operators
\begin{equation}\label{e.oper}
[H_\omega \psi](n) = \psi(n+1) + \psi(n-1) + V_\omega(n) \psi(n)
\end{equation}
in $\ell^2(\Z)$ with a random potential $V_\omega : \Z \to \R$. For simplicity we will focus on the case of bounded potentials throughout the paper. The essential features of the problems we study already arise in this setting.

Let us be more specific about the choice of the random potential. Randomness usually refers to the statement that the $\{ V_\omega(n) \}_{n \in \Z}$ are \emph{random variables}, defined on a probability space $(\Omega,\mu)$. Beyond that it is quite customary to require these random variables to be \emph{independent}. Finally, the standard \emph{Anderson model} arises once one assumes in addition that the random variables are \emph{identically distributed}.

In the case of the Anderson model, that is, when the $\{ V_\omega(n) \}_{n \in \Z}$ are independent and identically distributed, it is well known that the spectrum of $H_\omega$ is $\mu$-almost surely independent of the realization of the potential, that is, the spectrum is \emph{non-random} in this sense. Moreover, denoting the common distribution of the $\{ V_\omega(n) \}_{n \in \Z}$ by $\nu$ (so that $\mu = \nu^\Z$), the almost sure spectrum can be explicitly described as follows (see, e.g., \cite{D17, DF21, K08}):
\begin{equation}\label{e.asspec}
\sigma(H_\omega) = [-2,2] + \mathrm{supp} \, \nu =: \Sigma_\nu^\mathrm{AM} \quad \mu-\text{almost surely}.
\end{equation}
Here, $\mathrm{supp} \, \nu$ denotes the \emph{topological support} of $\nu$. In particular, each connected component of the almost sure spectrum is an interval of length at least $4$.

As an obvious immediate consequence of \eqref{e.asspec} we see that Cantor-type structures are impossible for the almost sure spectrum $\Sigma_\nu^\mathrm{AM}$ of the Anderson model. Given this well-known state of affairs, the reader may now observe that the title of this paper is intentionally provocative, and the immediate reaction of anyone familiar with the Anderson model is to say: ``Yeah, of course!''

Alas, we want to explain and emphasize in this paper that the answer is far from obvious and, in the setting we will consider, actually not known. To arrive at this setting we will drop the assumption of identical distribution. The potential in this case is still given by a collection of independent random variables and hence can and should be considered to be random. Potentials of this kind are naturally of interest, for example when modeling (perhaps small) random perturbations of a given deterministic potential. Indeed, models of this kind have been considered before, with the same motivation, but usually in the more tractable case of a periodic background potential, in which case one is able to model a crystal with random impurities; see, for example, \cite{AK19, CH94, DFG, GK13, KSS98, KSS98b, K99, ST20, V02, WW} for a partial list of papers studying this model. However, in this special case the answer to the question in the title of the paper is still ``yes'' for obvious reasons: the deterministic spectrum already contains intervals, which are typically only enlarged by the random perturbation.

The situation becomes much more difficult if the spectrum of the deterministic background model does not contain any intervals. In this case, the expert in random Schr\"odinger operators may still suspect a positive answer, but we challenge said expert to come up with a proof that works in complete generality.

Having set up the problem we want to study, let us now make it more concrete. First, observe that for an arbitrary deterministic background potential $V_\mathrm{bg}$ and an i.i.d.\ perturbation $\{ V^\mathrm{AM}_\omega \}_{\omega \in \Omega}$, there is no general reason for the spectrum of the operator $H_\omega$ with potential
\begin{equation}\label{e.potentialsum}
V_\omega = V_\mathrm{bg} + V^\mathrm{AM}_\omega
\end{equation}
to be the same for $\mu$-almost every $\omega \in \Omega$. Thus, we may want to restrict to a situation where such a non-random spectrum exists. Secondly, since we want to consider a case where the deterministic spectrum is a Cantor set (to make the question in the title of the paper as challenging as we want it to be), and Cantor spectra are most commonly observed for almost periodic potentials (the literature is vast; see, e.g., \cite{ABD, D17, E92, M81}), we may and should consider the case of an almost periodic $V_\mathrm{bg}$. A side benefit of doing so is that in this case, there indeed is a well defined non-random spectrum, as we will explain later.

Thus, we arrive at the following:
\begin{prob}\label{prob.1}
If $V_\mathrm{bg}$ is almost periodic and $\{ V^\mathrm{AM}_\omega(n) \}_{n \in \Z}$ are i.i.d.\ random variables with a common compactly supported single-site distribution $\nu$ that is non-degenerate (i.e., we have $\# \mathrm{supp} \, \nu \ge 2$), is it true that the almost sure spectrum of the operator $H_\omega$ given by \eqref{e.oper} and \eqref{e.potentialsum} must contain an interval?
\end{prob}

Of course the non-degeneracy assumption in Problem~\ref{prob.1} is necessary, as the answer is otherwise clearly negative (just take an almost periodic background potential for which Cantor spectrum has been verified and $\nu = \delta_0$).

Our purpose in this paper is twofold. On the one hand we wish to propose and advertise some problems centered around the structure of the spectrum of a random Schr\"odinger operator (when identical distribution fails) we believe are both open and interesting. Problem~\ref{prob.1} is the first such problem. However, it is possible that solving it completely is within reach, given current technology, and hence we will formulate more challenging problems below. On the other hand we wish to prove a result that answers Problem~\ref{prob.1} under additional assumptions. In doing so we also propose some ideas, and a road map, related to a possible eventual complete solution of  Problem \ref{prob.1}.

Let us first formulate the more challenging problems. We can ask the same question as in Problem~\ref{prob.1}, but without assuming the almost periodicity of $V_\mathrm{bg}$. One reason is that there are non-almost periodic background potentials known for which the spectrum is a Cantor set, such as Fibonacci \cite{S87, S89} and more general potentials: Sturmian potentials \cite{BIST89}, aperiodic potentials generated by primitive substitutions \cite{BG93, L02, LTWW02}, or elements of aperiodic Boshernitzan subshifts \cite{DL06a, DL06b}. Once an i.i.d.\ perturbation of such a background potential is considered, operators would arise that can for example model a one-dimensional quasicrystal with random impurities, which is certainly a model of interest. Another reason is that there may be other choices of background potentials for which the spectrum of the deterministic operator does not contain any intervals, but for which even ergodicity fails. In the latter case one should be careful about the expectation that there is a non-random spectrum after the addition of an i.i.d.\ perturbation. Nevertheless, note that the spectrum of $H_\omega$ contains an interval if and only if the \emph{essential} spectrum of $H_\omega$ contains an interval, and the latter set is almost surely non-random even in the absence of ergodicity due to Kolmogorov's zero-one law. We explain this observation in more detail in the appendix; see Theorem~\ref{t.essspect}.

\begin{prob}\label{prob.2}
If a bounded $V_\mathrm{bg} : \Z \to \R$ is given and $\{ V^\mathrm{AM}_\omega(n) \}_{n \in \Z}$ are i.i.d.\ random variables with a common compactly supported single-site distribution $\nu$ that is non-degenerate (i.e., we have $\# \mathrm{supp} \, \nu \ge 2$), is it true that with $\mu = \nu^\Z$, the set $\Sigma$, which for $\mu$-almost every $\omega$ is equal to the essential spectrum of the operator $H_\omega$ given by \eqref{e.oper} and \eqref{e.potentialsum}, must contain an interval?
\end{prob}

\begin{remark}\label{r.problem1.2}
As mentioned above, we limit our attention in this paper to the case of bounded potentials. On the one hand, this is done for the sake of simplicity. On the other hand, the way Problem~\ref{prob.2} is formulated, it would be possible to find counterexamples without the boundedness assumption on $V_\mathrm{bg}$  (cf.~Appendix~\ref{app.Unbounded}) and hence boundedness of $V_\mathrm{bg}$ is not only a convenient assumption, but it also cannot be dropped.
\end{remark}

\begin{remark}
Another open question related to Problems \ref{prob.1} and \ref{prob.2} that we would like to mention briefly is whether the spectrum of \eqref{e.oper} with potential given by \eqref{e.potentialsum} with a periodic $V_\mathrm{bg}$ must be a finite union of intervals. This question seems to have beautiful connections to a dynamical question on a description of the hyperbolicity locus of actions of finitely generated semigroups of $\mathrm{SL}(2, \mathbb{R})$ matrices discussed in \cite{ABY, Y}. We do not discuss this here, since it would lead us too far from the main subject of the paper, but would like to refer the interested reader to an upcoming publication \cite{WW} for details. We also mention that restrictions on the gap structure are imposed by the gap labeling theorem \cite{Bell}.
\end{remark}

At this point one may as well subsume the value $V_\mathrm{bg}(n)$ in the distribution of $V_\omega(n)$, namely it will be given by $V_\mathrm{bg}(n) + \nu$. In other words, one is dealing with an $n$-dependent distribution at site $n$, which in turn suggests how the model can be generalized: the distribution at site $n$ can be of a general nature and does not need to take the form $V_\mathrm{bg}(n) + \nu$. We arrive at

\begin{prob}\label{prob.3}
If the $\{ V_\omega(n) \}_{n \in \Z}$ are independent random variables with single-site distributions $\{ \nu_n \}_{n \in \Z}$ with supports that are uniformly bounded and variations that are uniformly bounded away from zero, is it true that with $\mu = \prod_{n \in \Z} \nu_n$, the set $\Sigma$, which for $\mu$-almost every $\omega$ is equal to the essential spectrum of the operator $H_\omega$ given by \eqref{e.oper}, must contain an interval?
\end{prob}

Summarizing this discussion, we believe that Problems~\ref{prob.1}--\ref{prob.3} are both open and interesting, and they are progressively harder to affirm (resp., easier to show to admit counterexamples).
%In fact, the boundedness assumption in Problem~\ref{prob.2} is made to exclude easy counterexamples. In the same vein one should impose some non-triviality condition in Problem~\ref{prob.3}. In that setting, however, it is somewhat less clear what the most natural condition is, but for the sake of definiteness, one could assume that there is a fixed compact subset of $\R$ that contains $\mathrm{supp} \, \nu_n$ for every $n \in \Z$. Weaker assumptions are possible and still render an interesting problem, but once they are too weak, it will again be possible to exhibit easy counterexamples.

\begin{remark}
As a final general remark pertaining to the non-stationary case, we point out that the issue studied here, namely the topological structure of the almost sure (essential) spectrum, is as of yet poorly understood, whereas a different, but equally interesting and important, issue is much better understood: it was shown in \cite{GK} that in the setting of Problem~\ref{prob.3}, the corresponding Schr\"odinger operators exhibit Anderson localization in both the spectral and the dynamical version, that is, the spectral type is almost surely pure point, the corresponding eigenfunctions decay exponentially, and initially localized wave packets do not travel off to infinity under the associated Schr\"odinger time evolution. Related earlier work can be found in \cite{KS80}, where a method was developed that can prove localization in the non-stationary case with absolutely continuous distributions under some additional assumptions. Moreover, in \cite{Kl3} the continuum ``crooked'' Anderson model (which can be considered an analog of the non-stationary random case) is investigated, and localization is proved under a H\"older continuity assumption on the distributions.
\end{remark}

\bigskip

Having formulated open problems that we hope will intrigue some of the readers and generate some research activity, let us now turn to our partial answer to Problem~\ref{prob.1}. We will exhibit a collection of almost periodic background potentials $V_\mathrm{bg}$ for which we will provide an affirmative answer to the problem. This collection is commonly referred to as ``perturbatively small quasi-periodic potentials with analytic sampling function and Diophantine frequency.'' For such potentials, it is well known that the resulting spectrum is generically a Cantor set. What we do here is add a non-degenerate i.i.d.\ perturbation and prove the existence of an interval in the almost sure spectrum of the resulting Schr\"odinger operator. We begin with the existence of the almost sure spectrum, which holds under quite weak assumptions.

\begin{theorem}\label{t.almostsurespectrum}
Suppose $T : X \to X$ is a minimal homeomorphism of a compact metric space $X$ and $f \in C(X,\R)$. Suppose further that $\nu$ is a compactly supported single-site distribution. Then there is a compact $\Sigma_\nu \subset \R$ such that for every $x \in X$ and $\mu = \nu^\Z$-almost every $\omega$, we have $\sigma(H_\omega) = \Sigma_\nu$, where $H_\omega$ is given by \eqref{e.oper} and \eqref{e.potentialsum} with $V_\mathrm{bg}(n) = f(T^n x)$.
\end{theorem}

\begin{remark}\label{t.supporttheorem}
(a) Minimality means that every $T$-orbit is dense, that is, we have
$$
X = \overline{ \{ T^n x : n \in \Z \} } \quad \text{ for every } x \in X.
$$

(b) Every almost periodic background potential can be written in this way. Indeed, $X$ can be chosen to be the hull of the given almost periodic potential, obtained by the $\ell^\infty$-closure of the set of its translates, $T$ is the shift, and $f$ is the evaluation at the origin.
\end{remark}

We have the following monotonicity result, which in the case of zero background is known as \emph{Kotani's support theorem}; compare \cite{DF21, K85}.

\begin{theorem}\label{t.supporttheorem}
Suppose $T : X \to X$ is a minimal homeomorphism of a compact metric space $X$ and $f \in C(X,\R)$. Using the notation from Theorem~\ref{t.almostsurespectrum}, we have
\begin{equation}\label{e.monotonicity}
\mathrm{supp} \, \nu_1 \subseteq \mathrm{supp} \, \nu_2 \quad \Rightarrow \quad \Sigma_{\nu_1} \subseteq \Sigma_{\nu_2}.
\end{equation}
\end{theorem}

\begin{remark}\label{r.redtobernoulli}
Theorem~\ref{t.supporttheorem} permits us to restrict our attention to the case where the i.i.d.\ portion of the random potential is given by the Bernoulli-Anderson model. That is, if one is able to identify intervals in $\Sigma_{\nu_1}$ for some $\nu_1$ with $\# \mathrm{supp} \, \nu_1 = 2$, then any $\nu_2$ whose topological support $\mathrm{supp} \, \nu_2$ contains the two elements of $\mathrm{supp} \, \nu_1$ will be such that $\Sigma_{\nu_2}$ contains an interval as well. Put differently, given an arbitrary $\nu_2$, we choose two suitable elements of $\mathrm{supp} \, \nu_2$, assign non-zero probabilities to each, and study the resulting $\nu_1$.
\end{remark}

\begin{theorem}\label{t.main1}
Given $d \in \N$, $f \in C^\omega(\T^d,\R)$, and $\alpha \in \T^d$ Diophantine, there exists $c_1 > 0$ such that for $0 \le c < c_1$ the following folds. There is $\lambda_0 > 0$ such that for every compactly supported and non-degenerate single-site distribution $\nu$ with
\begin{equation}\label{e.ssdassumption}
\inf \{ |E-\tilde E| : E, \tilde E \in \mathrm{supp} \, \nu, \; E \not= \tilde E \} < \lambda_0,
\end{equation}
 the almost sure spectrum $\Sigma_\nu$ associated with $H_\omega$ given by \eqref{e.oper} and \eqref{e.potentialsum} with $V_\mathrm{bg}(n) = c f(\theta + n\alpha)$ contains an interval.
\end{theorem}

\begin{remark}\label{rem.mainthm}
(a) Here, $\alpha \in \T^d$ is called \emph{Diophantine} if there exist constants $\tau, K > 0$ such that for every $m \in \Z^d \setminus \{ 0 \}$, we have
$$
\mathrm{dist} \, (m \cdot \alpha , \Z) \ge \frac{K}{|m|^\tau}.
$$
The relevance of our assumptions on $f$, $\alpha$, and $c_1$ is that in the proof of Theorem~\ref{t.main1} we want to apply the reducibility result from \cite{HA09}, which was established under these assumptions. Indeed we only apply that reducibility statement at the top of the spectrum, where the rotation number vanishes; see specifically Proposition~\ref{p.conjtoconstant} below. This is the only place where these assumptions enter in our discussion, and hence any similar reducibility result that is applicable at the top of the quasi-periodic spectrum and holds under weaker assumptions will be sufficient for our proof.

(b) A single-site distribution $\nu$ is considered to be \emph{non-degenerate} if $\# \mathrm{supp} \, \nu \ge 2$.

(c) The condition \eqref{e.ssdassumption} holds as soon as $\mathrm{supp} \, \nu$ is not discrete, for example when $\nu$ is not pure point. If, on the other hand, $\mathrm{supp} \, \nu$ is discrete, then \eqref{e.ssdassumption} requires at least two of its elements to be sufficiently close.

(d) The statement about the existence of the almost sure spectrum $\Sigma_\nu$ holds in greater generality and will be proved under the appropriate set of assumptions in Appendix~\ref{sec.almostsurespectrum}. On the other hand, our proof of the statement that $\Sigma_\nu$ has non-empty interior does currently require the setting described in Theorem~\ref{t.main1}. It would certainly be of interest to prove this conclusion under weaker assumptions.

(e) For more recent work studying a combination of quasi-periodicity and randomness, see \cite{BMOT, BePo, CDK}.

(f) The key idea in the proof of Theorem~\ref{t.main1} is that for a fixed phase $\theta$ and every energy in a small ``left''-neighborhood of the top of the almost sure spectrum, one may generate realizations of the random potential so that the resulting operator admits an exponentially decaying eigenvector corresponding to the energy in question. As a result, each of these energies must be an eigenvalue of that operator, and consequently belong to the almost sure spectrum.
\end{remark}

The proof of Theorem~\ref{t.main1} will require some preparatory work that actually concerns the classical Anderson model. While it would strictly speaking not be necessary to first consider the case of zero background potential, we have decided to proceed in two steps for the following reasons:

\begin{enumerate}

\item The results we prove for the classical Anderson model are interesting in their own right and hence deserve to be formulated and proved in this setting.

\item The extension from zero background to quasi-periodic background is then easier to process, as it will become clear that what is happening in the proof of Theorem~\ref{t.main1} is really the analog of a well-motivated and more classical problem associated with the Anderson model.

\end{enumerate}

We will employ the reduction to the Bernoulli case described in Remark~\ref{r.redtobernoulli}. Thus, we consider the classical Anderson model in this special case as well. Without loss of generality (i.e. up to a shift in energy, which of course does not affect the topological structure of the spectrum because it results in a mere translation), we will consider the case where the two values are given by $0$ and $\lambda > 0$. By \eqref{e.asspec}, the almost sure spectrum is then given by
\begin{equation}\label{e.asspectbernoulli}
\Sigma(H_\omega) = [-2,2] \cup [-2 + \lambda, 2 + \lambda],
\end{equation}
and in particular the top of the spectrum is equal to $2 + \lambda$. The top of the spectrum will play a crucial role in this paper. We will refer to it as the \emph{ground state energy} and investigate the question of whether there is a corresponding \emph{ground state}. In classical settings (i.e., for typical atomic models; compare, e.g., \cite[Chapter~10]{T14}) the ground state is always a strictly positive bona fide eigenfunction, but square-summability does not always have to hold. As we discuss in  Appendix \ref{sec.basicsofGS}, a more widely applicable notion involves Gilbert-Pearson's subordinacy property. We will prove the following result for the classical Bernoulli-Anderson model, which shows that there aren't even any decaying positive solutions at the top of the almost sure spectrum for any realization of the random potential, and which will be our first step towards developing the techniques for the proof of Theorem~\ref{t.main1}.

\begin{theorem}\label{t.eigenvectors}
Suppose that $\mathrm{supp} \, \nu = \{ 0, \lambda \}$ for some $\lambda > 0$. Given an arbitrary $\omega \in \Omega = (\mathrm{supp} \, \nu)^\Z$ and a solution $u$ of the difference equation
\begin{equation}\label{e.eveAM}
u(n+1) + u(n-1) + V^\mathrm{AM}_\omega (n) u(n) = E u(n), \quad n \in \Z
\end{equation}
at energy $E = 2 + \lambda$ that has strictly positive entries, we must have
\begin{equation}\label{e.nondecayprop}
\liminf_{n\to \infty} u(n) > 0 \quad  \text{or} \quad \liminf_{n \to -\infty} u(n) > 0.
\end{equation}
\end{theorem}

\begin{remark}
(a) As was mentioned above, the energy $E = 2 + \lambda$ is the top of the spectrum of $H^\mathrm{AM}_\omega$ for $\mu = \nu^\Z$-almost every $\omega \in \Omega$, and Theorem~\ref{t.eigenvectors} shows that for none of these $\omega$'s, this energy is actually an eigenvalue. In other words, for typical realizations of the Bernoulli-Anderson model, there is no square-summable ground state.

(b) However, the conclusion from part (a) has a much more direct proof, which was pointed out to us by Svetlana Jitomirskaya: Suppose to the contrary that there exists an  $\omega \in (\mathrm{supp} \, \nu)^\Z$ so that $2 + \lambda$ is an eigenvalue of $H_\omega$. Denote by $\psi \in \ell^2(\Z)$ a corresponding normalized eigenvector, that is, $H_\omega \psi = (2 + \lambda) \psi$ and $\| \psi \| = 1$. Denote by $\omega_\mathrm{max}$ the sequence in $(\mathrm{supp} \, \nu)^\Z$ for which all entries are equal to $\lambda$. Then,
$$
2 + \lambda = \langle \psi, H_\omega \psi \rangle \le \langle \psi, H_{\omega_\mathrm{max}} \psi \rangle < 2 + \lambda.
$$
Here the first inequality holds since $H_\omega \le H_{\omega_\mathrm{max}}$ (which is obvious) and the second holds since $H_{\omega_\mathrm{max}}$ has no eigenvalues and obeys $H_{\omega_\mathrm{max}} \le 2 + \lambda$ (both of which are well known). This contradiction shows that our assumption was wrong and $H_\omega$ does not have $2 + \lambda$ as an eigenvalue.

(c) Given the discussion in part (b), the main value of Theorem~\ref{t.eigenvectors} lies in establishing the stronger non-decay property and in introducing some ideas and methods that eventually will lead to the proof of our main result, Theorem~\ref{t.main1}.

(d) Soshnikov \cite{Sosh} and Bjerkl\"ov  \cite{Bje1, Bje2} have studied certain strongly coupled quasi-periodic models for which they showed that, for a suitably chosen phase, the top of the spectrum is an eigenvalue admitting an exponentially decaying eigenvector. Some of these models are known to display Anderson localization. We therefore see that the two most popular localized Schr\"odinger operators in one space dimension behave differently in this particular regard.
\end{remark}

As a complementary result, let us then show that for sufficiently small values of $\lambda$, there are some atypical realizations of the random potential (i.e., they do not belong to the full-measure set for which the spectrum coincides with $\Sigma^\mathrm{AM}_\nu$) for which the difference equation corresponding to the energy at the top of the spectrum actually does admit a square-summable ground state. Indeed this happens for ground state energies from a non-degenerate interval:

\begin{theorem}\label{t.groundstates}
For any sufficiently small $\lambda > 0$, there exists an interval $(2+\lambda-\varepsilon, 2+\lambda)$, $\varepsilon>0$, such that for any $E\in (2+\lambda-\varepsilon, 2+\lambda)$, there exists a bi-infinite sequence of $0$'s and $\lambda$'s such that $E$ is the maximal energy in the spectrum of the Schr\"odinger operator with potential given by that sequence, $E$ is an eigenvalue of this operator, and the corresponding eigenfunction is exponentially decaying and has strictly positive components.
\end{theorem}

\begin{remark}
While the realizations that are discussed in Theorem \ref{t.groundstates} are atypical (i.e. of zero $\nu^{\Z}$-measure), they form a set that is not too small. In particular, it has positive Hausdorff dimension for each $E$ close to $2+\lambda$, see Proposition \ref{p.Hausdorff} below.
\end{remark}

%To the best of our knowledge, this statement about the classical Bernoulli-Anderson model is new.

A consequence of this result, which at first sight appears to have no importance in this context, is that the interval $(2+\lambda-\varepsilon, 2+\lambda)$ must belong to the almost sure spectrum $\Sigma^\mathrm{AM}_\nu$. Of course we knew that already from \eqref{e.asspectbernoulli}, but the main point is that this particular way of seeing that the interval in question belongs to the almost sure spectrum will extend to the case of quasi-periodic background and hence provide us with the key mechanism we will use to establish Theorem~\ref{t.main1}; compare Remark~\ref{rem.mainthm}.(f).

\section*{The Organization of This Paper} 

In Section~\ref{sec.absenceofGS} we establish the absence of decaying generalized ground states in Bernoulli Anderson models for typical realizations as stated in Theorem~\ref{t.eigenvectors}. 

Then we show the presence of ground states in Bernoulli Anderson models for atypical realizations in Section~\ref{sec.presenceofGS}. In particular, we prove Theorem~\ref{t.groundstates}, but we also show that the set of these atypical realization, while of measure zero, is not small in the sense that it has positive Hausdorff dimension (see Proposition~\ref{p.Hausdorff}). 

In Section~\ref{sec.intervalsinSigma} we prove our main result, Theorem~\ref{t.main1}, by first establishing an analog of Theorem~\ref{t.groundstates} for suitable quasi-periodic background potentials (see Theorem~\ref{t.groundstatesforquasiplus}), and then deriving Theorem~\ref{t.main1} from it. 

We conclude the paper with several appendices that contain material that elucidates the material in the main body of the paper further, as well as tools needed there which are likely known but for which we could not find a reference that presents them in the exact form needed. 

In Appendix~\ref{app.Unbounded} we exhibit an example of an  unbounded background potential such that no random perturbation with compactly single-site distributions can produce any intervals in the spectrum. 

We discuss the existence of the almost sure spectrum (resp., essential spectrum) in Appendix~\ref{sec.almostsurespectrum}. In particular, we prove Theorems~\ref{t.almostsurespectrum} and \ref{t.supporttheorem}, and we also prove the existence of an almost sure essential spectrum under very weak assumptions (see Theorem~\ref{t.essspect}). 

Finally, we provide a discussion of generalized ground states in Appendix~\ref{sec.basicsofGS}. Here we consider solutions to the associated difference equation at energy given by the top of the spectrum (the \emph{ground state energy}) that are smallest among all solutions in a suitable sense. We show that the notion of subordinacy due to Gilbert and Pearson provides for a suitable smallness concept that, when combined with a positivity property, allows us to prove the existence of \emph{generalized ground states} in this sense in a very general setting. Specializing then to the small quasi-periodic background case of interest to us, we explain how reducibility results established via KAM methods can single out the generalized ground state as the unique bounded solution at the top energy.

\section*{Acknowledgments}

We would like to thank Svetlana Jitomirskaya and Milivoje Lukic for helpful conversations, and Jean Bellissard for providing several useful references. Also, we would like to thank Victor Kleptsyn for many discussions that partly motivated this paper, and for communicating  Theorem \ref{t.essspect}  to us.

\section{Absence of Ground States in Bernoulli Anderson Models for Typical Realizations}\label{sec.absenceofGS}

In this section we prove Theorem~\ref{t.eigenvectors}. Thus, we consider the Bernoulli Anderson model and show that at the top of the almost sure spectrum, the difference equation \eqref{e.eveAM} does not admit any decaying positive solutions. %Moreover, we show that the generalized ground state is not an actual ground state in that any positive solution must satisfy the non-decay statement \eqref{e.nondecayprop}.

\begin{proof}[Proof of Theorem~\ref{t.eigenvectors}]
Define
$$
B=\begin{pmatrix}
    2 & -1 \\
    1 & 0 \\
  \end{pmatrix}
$$
and
$$
  \Pi_a=\begin{pmatrix}
    2+a & -1 \\
    1 & 0 \\
  \end{pmatrix}=\begin{pmatrix}
    1 & a \\
    0 & 1 \\
  \end{pmatrix}B.
$$
The unit vector $\begin{pmatrix}
                 \cos \psi \\
                 \sin\psi \\
               \end{pmatrix}$ is parallel to the vector $\begin{pmatrix}
                 \cot\psi \\
                  1\\
               \end{pmatrix}$. An application of the matrix $\Pi_a$ to the vector $\begin{pmatrix}
                 z \\
                  1\\
               \end{pmatrix}$ gives a vector parallel to $\begin{pmatrix}
                 2+a-\frac{1}{z} \\
                  1\\
               \end{pmatrix}$, so let us consider the map $F_a:\mathbb{R}^+\to \mathbb{R}$ given by
$$
F_a(z)=2+a-\frac{1}{z},
$$
%Let us denote $F_a:\mathbb{R}^+\to \mathbb{R}$, $F_a(x)=2+a-\frac{1}{x}$.
 Consider the transfer matrices
$$
  \Pi_0=\begin{pmatrix}
    E & -1 \\
    1 & 0 \\
  \end{pmatrix}\  \text{and}\   \Pi_{\lambda}=\begin{pmatrix}
    E-\lambda & -1 \\
    1 & 0 \\
  \end{pmatrix}.
$$
Set the energy to be $E = 2+\lambda$. The corresponding projective maps act on the points in the first quadrant (with the vectors in the first quadrant parameterized by the cotangent of the argument) by
$$
F_\lambda : x \mapsto 2 + \lambda - \frac{1}{x} \ \ \text{and}\ \ \ F_0 : x \mapsto 2 - \frac{1}{x}.
$$

Let  $\omega \in \{0, \lambda\}^{\mathbb{Z}}$ be given and suppose that $u$ is a solution of \eqref{e.eveAM} (where $V_\omega^\mathrm{AM}(n)=\omega_n\in \{0, \lambda\}$) with $E = 2 + \lambda$ that has strictly positive entries. Set $x_n=\frac{u(n+1)}{u(n)}$. Notice that the projective map generated by $\Pi_{\omega_k}|_{E=2+\lambda}$ corresponds to $F_0$ if $\omega_k = \lambda$ and to $F_\lambda$ if $\omega_k=0$. Also, if $A \in \mathrm{SL}(2, \mathbb{R})$, $v$ is a unit vector that corresponds to the point $x \in \T$ on the unit circle, and $f : \T \to \T$ is the projective map that corresponds to the linear map $A$, then it is known that $|f'(x)|=\frac{1}{\|Av\|^2}$ (indeed, due to the existence of the polar decomposition it is enough to check this for a diagonal matrix $A$, in which case the calculation becomes explicit). Since the parameterizations by the cotangent of the argument in a compact arc inside of the first quadrant only adds a bounded multiplicative factor to the derivative of a map, in view of the discussion in Appendix~\ref{sec.basicsofGS} it is enough to show that if $\{x_n\}_{n\in \mathbb{Z}}$ is a sequence of positive real numbers such that for some sequence $\{\eta_n\}\in \{0, \lambda \}^\mathbb{Z}$ (specifically, $\eta_n=\lambda-\omega_n$), we have $F_{\eta_n}(x_n)=x_{n+1}$, then
\begin{equation}\label{e.nondecaygoal1}
\limsup_{n\to \infty}(F_{\eta_n}\circ \ldots \circ F_{\eta_1}\circ F_{\eta_0})'(x_0)<\infty
\end{equation}
or
\begin{equation}\label{e.nondecaygoal2}
\limsup_{n\to \infty}(F_{\eta_{-n}}^{-1}\circ \ldots \circ F_{\eta_{-1}}^{-1})'(x_0)<\infty.
\end{equation}
Notice that for any $x>0$, $(F_\lambda)'(x)=F_0'(x)=\frac{1}{x^2}$. The map $F_0$ has one fixed point, namely $F_0(1)=1$, and that fixed point is semi-stable, that is, $F_0'(1)=1$, and for any $x \in (0,1)$, we have $F_0^{-k}(x)\to 1$ as $k\to +\infty$, and for any $x>1$ we have $F^k_0(x)\to 1$ as $k\to +\infty$.

The map $F_\lambda$ has two fixed points, $1+\frac{\lambda}{2}\pm \sqrt{\lambda+\frac{\lambda^2}{4}}$. One of them, $x_{att}=1+\frac{\lambda}{2}+ \sqrt{\lambda+\frac{\lambda^2}{4}}$, is attracting, and another one, $x_{rep}=1+\frac{\lambda}{2}- \sqrt{\lambda+\frac{\lambda^2}{4}}$, is repelling.%\marginpar{Turn on figure when creating pdf.}

\begin{figure}[htp]
\centering
 \includegraphics[width=0.7\textwidth]{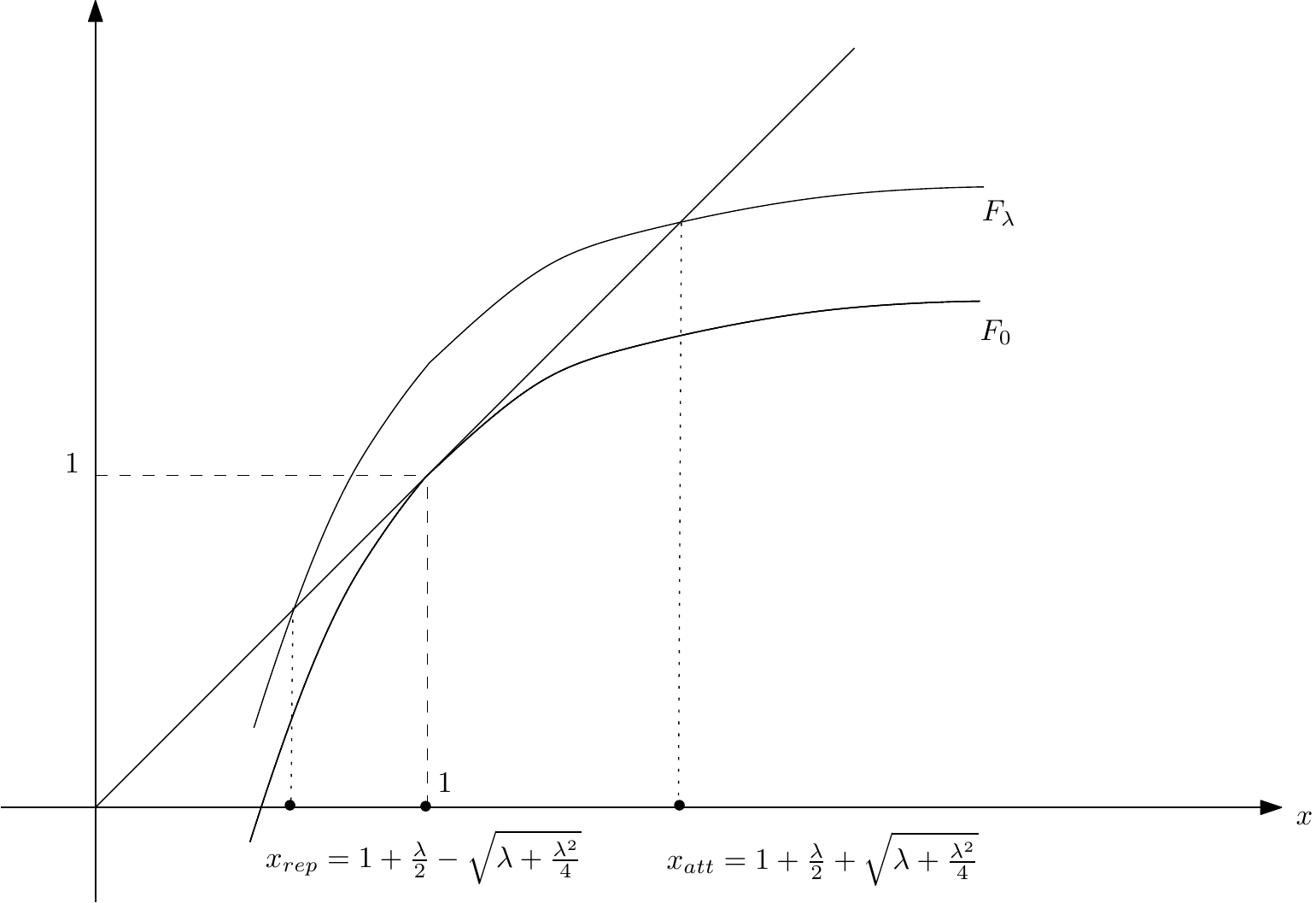}
 \caption{The maps $F_0$ and $F_\lambda$.}
\end{figure}

Let us consider now the possible cases of dynamical behavior of the iterates of $x_0$ under the sequence of maps $\{F_{\eta_n}\}_{n\in \mathbb{Z}}$, assuming $x_i>0$ for all $i\in \mathbb{Z}$.

If for some $n \in \mathbb{Z}$, we have $x_n\in (0, x_{rep})$, then $F_{\eta_n}(x_n)=x_{n+1}<x_n$, and eventually we would have $x_{n+k}<0$ for some $k>0$, which contradicts our assumption that $x_i>0$ for all $i\in \mathbb{Z}$.

Denote $I_\lambda=[1, x_{att}]$. Notice that $I_\lambda$ is a ``trapping interval'' for the maps $F_0$ and $F_\lambda$, that is, $F_0(I_\lambda) \subset I_\lambda$ and $F_\lambda(I_\lambda) \subset I_\lambda$. Since $\max_{x \in I_\lambda}(|F_0'(x)|, |F_\lambda'(x)|) \le 1$ if $x_n \in I_\lambda$ for some $n \in \mathbb{Z}$, then $x_{n+k} \in I_\lambda$ for all $k>0$, and
$$
\limsup_{k\to \infty}(F_{\eta_{n+k}}\circ \ldots \circ F_{\eta_n})'(x_n)<\infty,
$$
which implies that \eqref{e.nondecaygoal1} holds in this case.
%$$
%\limsup_{n\to \infty}(F_{\eta_{n}}\circ \ldots \circ F_{\eta_0})'(x_0)<\infty.
%$$
Similarly, the interval $\hat I_{\lambda}=[x_{rep}, 1]$ is a ``trapping region'' for the maps $F_0^{-1}$ and $F_\lambda^{-1}$, and if $x_n\in \hat I_\lambda$ for some $n \in \mathbb{Z}$, then \eqref{e.nondecaygoal2} holds in this case.
%$$
%\limsup_{n\to \infty}(F_{\eta_{-n}}^{-1}\circ \ldots \circ F_{\eta_{-1}}^{-1})'(x_0)<\infty.
%$$
Finally, if $x_n>x_{att}$, then either $x_{n+k}<x_{att}$ for some $k>0$, and one of the cases above can be applied, or $x_{n+k}\ge x_{att}$ for all $k>0$, and since $|F_0'(x)|, |F_\lambda'(x)|<1$ if $x\ge x_{att}$, we have \eqref{e.nondecaygoal1}.
%$$
%\limsup_{n\to \infty}(F_{\eta_{n}}\circ \ldots \circ F_{\eta_0})'(x_0)<\infty.
%$$
This completes the proof of the theorem.
%
%Finally, assume that $u$ is an eigenvector of $H_\omega^\mathrm{AM}$ corresponding to the eigenvalue $2+\lambda$ (i.e., to the top of the spectrum). By Appendix~\ref{sec.basicsofGS}, $u$ must have strictly positive entries (up to a multiplicative constant), and hence $\liminf_{n\to \infty} u(n) > 0$ or $\liminf_{n\to \infty} u(-n) > 0$ by what we have just proved. This is of course incompatible with $u$ being an eigenvector as it means that $u$ cannot be square-summable; contradiction. Thus, $2 + \lambda$ is in fact not an eigenvalue of $H_\omega^\mathrm{AM}$, proving the first part of the statement.
\end{proof}

\section{Presence of Ground States in Bernoulli Anderson Models for Atypical Realizations}\label{sec.presenceofGS}

Knowing that the maximum of the almost sure spectrum of the Bernoulli Anderson model is never an eigenvalue (for any realization), and hence no ground state exists in the classical sense whenever this particular energy, $2 + \lambda$, happens to be the top of the spectrum of a specific realization of the random potential (which is true almost surely), one may wonder whether an energy strictly smaller than $2 + \lambda$ can be the top of the spectrum and admit a ground state in the classical sense, that is, whether there is a square-summable positive solution at that energy. Theorem~\ref{t.groundstates} asserts that this scenario indeed occurs for some realizations if $\lambda > 0$ is small enough and the energy in question is sufficiently close to $2 + \lambda$. We begin by proving this theorem and explore later in this section how large in the sense of Hausdorff dimension the set of realizations is for which this can be verified.

\begin{proof}[Proof of Theorem \ref{t.groundstates}]
Suppose that $\lambda >0$ is sufficiently small, in particular such that the following inequalities hold:
\begin{equation}\label{e.condlambda}
(1-\lambda)\sqrt{\lambda+\frac{\lambda^2}{4}}-\frac{\lambda}{2} -\frac{\lambda^2}{2}> 0\ \ \text{and}\ \ \ \sqrt{\lambda+\frac{\lambda^2}{4}} > \frac{3\lambda}{2}.
\end{equation}
Suppose that $E=2+\lambda-a$, where $a\in (0, \lambda)$ is small, $a\ll \lambda$. Consider the transfer matrices for this energy $E$:
$$
  \Pi_0=\begin{pmatrix}
    E & -1 \\
    1 & 0 \\
  \end{pmatrix}=\begin{pmatrix}
    2+\lambda-a & -1 \\
    1 & 0 \\
  \end{pmatrix}
$$
and
$$
\Pi_{\lambda}=\begin{pmatrix}
    E-\lambda & -1 \\
    1 & 0 \\
  \end{pmatrix}=\begin{pmatrix}
    2-a & -1 \\
    1 & 0 \\
  \end{pmatrix}.
$$
The corresponding projective maps act on the points in the first quadrant (parameterized by the cotangent of the argument) by
$$
F_{\lambda-a}:x\mapsto 2+\lambda-a-\frac{1}{x}\ \ \text{and}\ \ F_{-a}:x\mapsto 2-a-\frac{1}{x}.
$$
The map $F_{\lambda-a}$ has two fixed points, the repelling $x_{rep} = x_{rep}(\lambda-a)$ and the attracting $x_{att}=x_{att}(\lambda-a)$. The map $F_{-a}$ does not have any fixed points, and for any $x>0$, we have $F_{-a}(x)<x$. Notice that for any $x>0$, we have $(F_{\lambda-a})'(x) = (F_{-a})'(x)=\frac{1}{x^2}$. For a small positive $\delta > 0$, denote by $I_{\lambda, a, \delta}$ the interval $I_{\lambda, a, \delta}=[1+\delta, x_{att}(\lambda-a)]$. % Also, denote by $I_{\lambda, a}$ the interval $I_{\lambda, a}=I_{\lambda, a, 0}=[1, x_{att}(\lambda-a)]$.

\begin{lemma}\label{l.inversesequence}
Given a small $\lambda>0$, for all sufficiently small $a>0$ and $\delta>0$ and for any $x\in I_{\lambda, a, \delta}$, we have $F_{-a}^{-1}(x)\in I_{\lambda, a, \delta}$ or $F_{\lambda-a}^{-1}(x)\in I_{\lambda, a, \delta}$ (it is possible that both statements hold).
\end{lemma}

\begin{proof}%[Proof of Lemma \ref{l.inversesequence}]
If $\lambda > 0$ is small, then $F_0(x_{att}(\lambda)) > F_{\lambda}(1)$. Indeed, this is equivalent to the first inequality in \eqref{e.condlambda}. By continuity, for any sufficiently small $a>0$ and $\delta>0$, we have $F_{-a}(x_{att}(\lambda-a))>F_{\lambda-a}(\delta)$. This implies that $I_{\lambda, a, \delta}\subset F_{-a}(I_{\lambda, a, \delta})\cup F_{\lambda-a}(I_{\lambda, a, \delta})$, and Lemma \ref{l.inversesequence} follows.
\end{proof}

Denote $\hat I_{\lambda, a, \delta}=[x_{rep}(\lambda-a), 1-\delta]$.

\begin{lemma}\label{l.forwardsequence}
Given a small $\lambda > 0$, for all sufficiently small $a>0$ and $\delta>0$ and for any $x\in \hat I_{\lambda, a, \delta}$, we have $F_{-a}(x)\in \hat I_{\lambda, a, \delta}$ or $F_{\lambda-a}(x)\in \hat I_{\lambda, a, \delta}$ (it is possible that both statements hold).
\end{lemma}

\begin{proof}%[Proof of Lemma \ref{l.forwardsequence}]
If $\lambda > 0$ is small, then $F_0^{-1}(x_{rep}(\lambda))<F_{\lambda}^{-1}(1)$. Indeed, this is equivalent to the inequality $\sqrt{\lambda+\frac{\lambda^2}{4}} > \frac{3\lambda}{2}$, which is satisfied if $\lambda>0$ is sufficiently small.

By continuity, for any sufficiently small $a>0$ and $\delta>0$ we have $F_{-a}^{-1}(x_{rep}(\lambda-a))<F_{\lambda-a}^{-1}(1-\delta)$. This implies that $\hat I_{\lambda, a, \delta}\subset F_{-a}^{-1}(\hat I_{\lambda, a, \delta})\cup F_{\lambda-a}^{-1}(\hat I_{\lambda, a, \delta})$, and Lemma~\ref{l.forwardsequence} follows.
\end{proof}

The next statement is clear.

\begin{lemma}\label{l.transit}
Given a small $\lambda>0$, for all sufficiently small $a>0$ and $\delta>0$ and any $x \in I_{\lambda, a, \delta}$, there exists $k>0$ such that $F_{-a}^k(x)\in \hat I_{\lambda, a, \delta}$.
\end{lemma}

Now, for any $x \in I_{\lambda, a, \delta}$, we can use Lemma~\ref{l.inversesequence} to generate a sequence of maps such that under the application of their inverses, the orbit of $x$ remains in $I_{\lambda, a, \delta}$ and the corresponding derivative increases exponentially, and Lemmas~\ref{l.forwardsequence} and \ref{l.transit} can be used to construct a forward sequence of iterates with similar properties. This gives a solution that is an eigenfunction that corresponds to the energy $E=2+\lambda-a$ and has positive entries (hence is a ground state).
\end{proof}

%\begin{remark}
It is clear that Theorem~\ref{t.groundstates} gives for each energy close to $2+\lambda$ uncountably many realizations of the random potential such that the energy in question is the ground state energy and an eigenvalue. In fact, one can show that this set of realizations is actually large in other ways. For example, the following statement holds:
\begin{prop}\label{p.Hausdorff}
In the setting of Theorem \ref{t.groundstates}, for any energy $E\in (2+\lambda-\varepsilon, 2+\lambda)$, the set of bi-infinite sequences of $0$'s and $\lambda$'s such that for each of them $E$ is the maximal energy in the spectrum of the associated Schr\"odinger operator has positive Hausdorff dimension.
\end{prop}
Notice that we do not need to specify a metric in the space of bi-infinite sequences as soon as it is one of the standard metrics, such as for example
$$d(\omega, \bar\omega)=\left\{
                                                                                      \begin{array}{ll}
                                                                                        0, & \hbox{if $\omega=\bar \omega$;} \\
                                                                                        2^{-k}, & \hbox{$\omega_i=\bar \omega_i$ if $|i|<k$ and either $\omega_k\ne \bar\omega_k$, or $\omega_{-k}\ne \bar\omega_{-k}$}
                                                                                      \end{array}
                                                                                    \right.
$$
or
$$
d(\omega, \bar\omega)=\sum_{n=-\infty}^{\infty}2^{-|n|}|\omega_n-\bar\omega_n|,
$$
since a H\"older homeomorphism sends a set of positive Hausdorff dimension to a set of positive Hausdorff dimension. Proposition \ref{p.Hausdorff} is implied by the following statement:
%\end{remark}

\begin{prop}\label{p.hausdorffdim}
Let $\Omega=\{0,1\}^\mathbb{N}=\{\omega=\omega_1\omega_2\ldots \omega_i\ldots\}$ be the space of sequences of $0$'s and $1$'s equipped with a standard metric, e.g
$$
d(\omega, \bar\omega)=\left\{
                                                                                      \begin{array}{ll}
                                                                                        0, & \hbox{if $\omega=\bar \omega$;} \\
                                                                                        2^{-k}, & \hbox{$\omega_i=\bar \omega_i$ if $i<k$ and $\omega_k\ne \bar\omega_k$.}
                                                                                      \end{array}
                                                                                    \right.
$$
Suppose $N\in \mathbb{N}$ and denote by $W_N$ the space of all finite sequences of $0$'s and $1$'s of length at most $N$ (empty sequence included). Denote by $W$ the space of all finite sequences of $0$'s and $1$'s, and suppose that a map $g:W\to W_N$ is given.

Let $\tilde W$ be the smallest space of finite words that has the following properties:

\vspace{4pt}

1) The empty word belongs to $\tilde W$;

\vspace{4pt}

2) If a finite word $a$ is from $\tilde W$, then the concatenation of $a$, $g(a)$, and $0$, as well as concatenation of $a$, $g(a)$, and $1$, belong to $\tilde W$.

\vspace{4pt}

Define $\tilde \Omega \subset \Omega$ to be the set of sequences $\omega\in \Omega$ such that for any $m\in \mathbb{N}$, there is a finite word $w\in \tilde W$ of length at least $m$ such that $\omega$ starts with $w$.

Then $\dim_H \tilde \Omega\ge \frac{1}{N+1} \dim_H \Omega.$ In particular, $\tilde \Omega$ has positive Hausdorff dimension.
\end{prop}

\begin{remark}
It is clear from the proof that an analog of Proposition \ref{p.hausdorffdim} for the space of bi-infinite sequences also holds, and can be proven in a very similar way.
\end{remark}

\begin{proof}[Proof of Proposition \ref{p.hausdorffdim}.]
Consider the oriented tree $T$ whose vertices are the elements of $\tilde W$, whose root is given by the empty word, and whose edges are given by $a\to b$ if and only if $b$ is a concatenation of $a$, $g(a)$, and $0$, or a concatenation of $a$, $g(a)$, and $1$. %$a$ is a prefix of $b$.
 An infinite branch of $T$ corresponds to a sequence from $\tilde \Omega$. Notice that any $\omega \in \tilde \Omega$ corresponds to exactly one infinite branch, and hence it can be obtained in a unique way as the limit of the sequence of finite words of the following form:
\begin{align*}
   & a_0=\emptyset \\
   & a_1=a_0 g(a_0)i_1, \ \ \text{where}\ \ i_1\in \{0,1\}\\
   & a_2=a_1 g(a_1)i_2, \ \ \text{where}\ \ i_2\in \{0,1\}\\
   & \ldots\\
   &  a_k=a_{k-1} g(a_{k-1})i_k, \ \ \text{where}\ \ i_k\in \{0,1\}\\
   &\ldots
\end{align*}
Consider the map $G:\tilde \Omega\to \Omega$ given by
$$
G(\omega)=i_1 i_2 \ldots i_k \ldots .
$$
Notice that $G$ is onto. Let us show that $G$ is H\"older continuous. Indeed, if %$G(\omega)=i_1i_2i_3\ldots i_k\ldots$ and $G(\bar\omega)=\bar i_1\bar i_2\bar i_3\ldots \bar i_k\ldots$, and
$d(G(\omega), G(\bar\omega))=2^{-k}$, then at most $k(N+1)$ first elements of $\omega$ and $\bar \omega$ can coincide, that is, $d(\omega, \bar \omega)\ge 2^{-k(N+1)}$. Therefore,
$$
d(G(\omega), G(\bar \omega)) = 2^{-k} = \left( 2^{-k(N+1)} \right)^{\frac{1}{N+1}} \le \left( d(\omega, \bar \omega) \right)^{\frac{1}{N+1}}.
$$
It is well known (and easy to check) that if for some metric spaces $X$ and $Y$ a map $G:X\to Y$ is $\beta$-H\"older continuous (i.e. $d_Y(G(x_1), G(x_2))\le C(d_X(x_1, x_2))^\beta$ for all $x_1, x_2\in X$ and some $C, \beta>0$), then $\text{dim}_H Y\le \frac{1}{\beta}\text{dim}_H X$. Hence $\dim_H \Omega \le (N+1) \dim_H \tilde \Omega$, and hence
$$
\dim_H \tilde \Omega \ge \frac{1}{N+1} \dim_H \Omega > 0,
$$
as claimed.
\end{proof}
\begin{proof}[Proof of Proposition \ref{p.Hausdorff}.]
In the proof of Theorem~\ref{t.groundstates}, in order to construct a sequence of iterates of the maps $F_{\lambda-a}$ and $F_{-a}$ that leads to the construction of the required eigenvector, one has to use Lemma~\ref{l.inversesequence} and Lemma~\ref{l.forwardsequence}. On some steps the choice of one of the two maps is determined, and in some cases it is arbitrary. The combinatorics is described by the setting of Proposition~\ref{p.hausdorffdim}, and hence the conclusion of Proposition~\ref{p.hausdorffdim} implies Proposition~\ref{p.Hausdorff}.
\end{proof}

We suspect that the set of realizations in question has Hausdorff dimension that tends to the full Hausdorff dimension as $\lambda, a \to 0$, but do not elaborate on that here.

\section{Intervals in the Spectrum of Random Perturbations of Quasi-Periodic Operators}\label{sec.intervalsinSigma}

Let us now consider the case of a quasi-periodic background potential. Specifically, we wish to prove Theorem~\ref{t.main1}. That is, we want to consider the case where the potential is given by a small analytic quasi-periodic sequence with Diophantine frequency and an Anderson part and show that the almost sure spectrum contains an interval. Recall that Theorem~\ref{t.supporttheorem} exhibits a monotonicity property of the almost sure spectrum that allows us to work with Bernoulli randomness.

Namely, let us consider the discrete Schr\"odinger operator \eqref{e.oper} with a potential of the form \eqref{e.potentialsum}, where the background potential is of the form $V_\mathrm{bg}(n) = c f(\theta + n \alpha)$ with an analytic function $f : \mathbb{T}^d\to \mathbb{R}$, a coupling constant $c\ge 0$, and $\theta, \alpha\in \mathbb{T}^d$, and the random piece $V_\omega^\mathrm{AM}$ is of Bernoulli type and generated by a single-site distribution $\nu$ with $\mathrm{supp} \, \nu = \{ 0, \lambda \}$ for some $\lambda > 0$. Let us denote the background operator (i.e., the Schr\"odinger operator with potential $V_\mathrm{bg}$, without the random piece) by $H_{\mathrm{bg},c}$.

%\begin{conj}\label{t.quasipermax}
%Suppose that $\alpha$ is Diophantine and the function $V:\mathbb{T}^d\to \mathbb{R}$ is analytic and sufficiently small. Let $E^*$ be the largest energy in the spectrum of the discrete Schr\"odiger operator with the quasiperiodic potential $\{V(\theta+n\alpha)\}_{n\in \mathbb{Z}}$. Then for every $\lambda > 0$ and every $\omega\in \{0, \lambda\}^{\mathbb{Z}}$, $E^*+\lambda$ is not an eigenvalue of the operator \eqref{e.quasiperplusBernoulli}.
%\end{conj}

\begin{theorem}\label{t.groundstatesforquasiplus}
Suppose that $\alpha$ is Diophantine, the phase $\theta$ is arbitrary, the function $f:\mathbb{T}^d\to \mathbb{R}$ is analytic, and the coupling constant $c>0$ is sufficiently small. Let $E^*_c$ be the largest energy in the spectrum of $H_{\mathrm{bg},c}$. For any sufficiently small $\lambda > 0$, there exists an interval $(E^*_c + \lambda - \varepsilon, E^*_c + \lambda)$, $\varepsilon>0$, such that for any $E\in (E^*+\lambda-\varepsilon, E^*+\lambda)$, there exists an $\omega\in \{0,\lambda\}^{\mathbb{Z}}$ %$\omega$
such that $E$ is the maximal energy in the spectrum of the corresponding operator $H_\omega$ whose potential term $V_\omega(n)$ is the sum of $V_\mathrm{bg}(n) = c f( \theta + n\alpha )$ and the $\{ 0, \lambda \}$-Bernoulli-Anderson term $V_\omega^\mathrm{AM}(n)$, $E$ is an eigenvalue of $H_\omega$, and the corresponding eigenfunction is exponentially decaying and has strictly positive components.
\end{theorem}

\begin{remark}
(a) Notice that Theorem~\ref{t.groundstatesforquasiplus} is an analog of Theorem~\ref{t.groundstates}, and this clarifies why we stated and proved the latter theorem -- really to exhibit the key ideas in a simpler setting.

(b) It would be interesting to see whether an analog of Theorem \ref{t.eigenvectors} holds in the present setting as well, that is, the top of the spectrum is such that no decaying solutions exist, provided that the potential is given by the sum of a suitable Bernoulli Anderson term and a small analytic quasi-periodic term with Diophantine frequency. Note that if we allow the quasi-periodic term to be large, then by the work of Soshnikov and Bjerkl\"ov mentioned before \cite{Bje1, Bje2, Sosh}, there will be a phase one can choose for the quasi-periodic term and a realization of the random term (namely the constant one) for which there exists an exponentially decaying solution of the difference equation with energy given by the top of the spectrum.
\end{remark}

We will start with the following statement.

\begin{prop}\label{p.conjtoconstant}
Suppose $\alpha$ is Diophantine and $f : \T^d \to \mathbb{R}$ is analytic. Then there is $c_0 > 0$ such that for every $c \in [0,c_0)$, the following holds:

Denote by $E^*_c$ be the top energy in the spectrum of the operator $H_{\mathrm{bg},c}$. There exists $Z_c : \T^d \to \mathrm{SL}(2, \mathbb{R})$ analytic such that
$$
\Pi_c(\theta+\alpha)=Z_c(\theta+\alpha) B Z_c^{-1}(\theta),
$$
where $B=\begin{pmatrix}
           2 & -1 \\
           1 & 0 \\
         \end{pmatrix}$
         and
         $\Pi_c(\theta)=\begin{pmatrix}
           E_c^*- cf(\theta) & -1 \\
           1 & 0 \\
         \end{pmatrix}$. Moreover, we have $Z_c(\theta) \to \begin{pmatrix}
           1 & 0 \\
           0 & 1 \\
         \end{pmatrix}$, uniformly in $\theta$, as $c \to 0$.
\end{prop}

\begin{proof}%[Proof of Proposition \ref{p.conjtoconstant}.]
%Proposition \ref{p.conjtoconstant} essentially follows from a combination of \cite[Theorem 1]{Amor} and \cite[Lemma 2]{Puig}. Indeed, \red{continue}.
By definition, $E^*_c$ is the ground state energy of $H_c$ and hence by general principles \cite{DF21}, the rotation number of the cocycle $(\alpha,\Pi_c)$ is zero. In particular, the rotation number of this cocycle is Diophantine with respect to $\alpha$ and therefore, by \cite[Theorem 1]{HA09}, the cocycle $(\alpha,\Pi_c)$ can be analytically conjugated to a constant parabolic cocycle $(\alpha, B_c)$ with $B_c$ being a parabolic matrix.
%$$
%J_c = \begin{pmatrix} 1 & j_c \\ 0 & 1 \end{pmatrix}.
%$$
Indeed, the constant cocycle one conjugates to cannot be non-identity elliptic because the rotation number is zero, it cannot be the identity by Corollary \ref{c.dyndefsetting}, % Remark~\ref{r.parabolicconstantmatrix},
and it cannot be hyperbolic because the energy in question belongs to the spectrum. Moreover, while \cite[Theorem 1]{HA09} yields a conjugacy that is defined on $(2 \T)^d$, we can feed the output of  \cite[Theorem 1]{HA09} into \cite[Lemma 2]{Puig}\footnote{Puig only states \cite[Lemma 2]{Puig} in the case $d = 1$, but the result and proof extend to the case of general $d \in \N$ in a straightforward way.} to obtain a conjugacy that is defined on $\T^d$. Therefore, there exists an analytic matrix valued function $\hat Z_c: \T^d \to \mathrm{SL}(2, \mathbb{R})$ such that
$$
\Pi_c(\theta+\alpha)=\hat Z_c(\theta+\alpha) B_c \hat Z_c^{-1}(\theta).
$$
As a consequence of \cite[Proposition~5]{HA09} (see also \cite[Proposition~3.1]{CCYZ19}), we get that $\hat Z(\theta)$ can be constructed in such  a way that $\hat Z(\theta) \to \begin{pmatrix}
           1 & 0 \\
           0 & 1 \\
         \end{pmatrix}$,
uniformly in $\theta$, as $c \to 0$.  Since $\Pi_c(\theta)\to B$ uniformly in $\theta$ as $c \to 0$, this implies that $B_c\to B$ as $c\to 0$. Since any parabolic matrix $B_c$ close to $B$ may be conjugated to $B$, and the conjugacy can be taken small if $B_c$ is close to $B$, Proposition~\ref{p.conjtoconstant} follows.
%Note first that $j_c$ (and hence $J_c$) depends continuously on $c \in [0,c_0)$. Moreover, for $c = 0$, the cocycle $(\alpha,\Pi_c)$ is (without any conjugacy) equal to the constant cocycle $(\alpha,B)$, which is orientation-preserving\marginpar{Check preserving/reversing and relate this to the sign of $j_0$.} in the fiber. Combining these two observations, we can conjugate the conjugacy provided by \cite[Theorem 1]{Amor} one more time and obtain that the cocycle $(\alpha,\Pi_c)$ can be analytically conjugated to the constant parabolic cocycle $(\alpha, B)$.
%
%Finally, the uniform convergence statement is a consequence of \cite[Proposition~5]{Amor}.
\end{proof}

For notational simplicity, we consider in the following the case $d = 1$. We invite the reader to verify that all statements and arguments extend to the case of general $d \in \N$.

For $\lambda > 0$, let us now introduce another cocycle, given by the matrices
$$
\Pi_{c, \lambda}(\theta)=\begin{pmatrix}
           E^*_c+\lambda-cf(\theta) & -1 \\
           1 & 0 \\
         \end{pmatrix}.
$$
Due to Johnson's theorem we know that this cocycle is uniformly hyperbolic. We will work in the coordinate system (provided by Proposition \ref{p.conjtoconstant}) in which the cocycle $\{\Pi(\theta)\}_{\theta\in \T}$ over the $\alpha$-rotation of the circle is constant.

We have
$$
B=Z^{-1}(\theta)\Pi_c(\theta)Z(\theta-\alpha).
$$
Denote $B_\lambda(\theta)=Z^{-1}(\theta)\Pi_{c, \lambda}(\theta)Z(\theta-\alpha)$. Notice that $B_\lambda$ converges to $B$ (as a function of $\theta$) in $C^\infty(\T, \SL(2, \mathbb{R}))$ as $\lambda\to 0$.

Let us parametrize unit vectors in the first quadrant by the cotangent of the argument. In that coordinate system the projective map defined by the matrix $B$ turns into
$$
F_0:(0, \infty)\to \mathbb{R}, \; F_0(x)=2-\frac{1}{x}.
$$
Notice that $F_0(1)=1$. Consider some compact interval containing $1$, for example $[1/2, 5]$. Let us extend the map $F_0$ to the map (we abuse notation by denoting the extended map by the same symbol) $F_0:\T\times [1/2, 5]\to \T\times \mathbb{R}$ by setting
$$
F_0(\theta, x)=\left(\theta +\alpha, 2-\frac{1}{x}\right).
$$
Let us now see how the projective action of the cocycle $\{\Pi_\lambda(\theta)\}$ looks like for small $\lambda>0$ in the same coordinate system.
 Notice that
 $$
 \Pi_{c, \lambda}(\theta)=\begin{pmatrix}
           1 & \lambda \\
           0 & 1 \\
         \end{pmatrix}\Pi(\theta),
$$
so
\begin{equation}\label{e.blambda}
 B_\lambda(\theta)=Z^{-1}(\theta)\begin{pmatrix}
           1 & \lambda \\
           0 & 1 \\
         \end{pmatrix}\Pi_c(\theta)Z(\theta-\alpha).
\end{equation}
Since $Z(\theta)$ is close to the identity matrix, from here we can see that for small $c>0$ and small $\lambda>0$, the projective map defined by $B_\lambda(\theta)$ in the coordinates given by the cotangent of the argument is defined on $[1/2, 5]$ and has the form
$$
x\mapsto 2-\frac{1}{x}+\varphi_\lambda(\theta, x),
$$
where for any $r\ge 0$, we have that $\|\varphi_\lambda(\theta,x)\|_{C^r}$ is small if $\lambda$ is small, and in fact $\|\varphi_\lambda(\theta,x)\|_{C^r}\le C\lambda$ for some $C>1$ independent of $\lambda$ (we suppress the dependence of $\varphi_\lambda$ on $c$ to simplify the notation). It is convenient to incorporate all these maps into a skew product
$$
F_{\varphi_\lambda}:\T\times [1/2, 5]\to \T\times \mathbb{R}, F_{\varphi_\lambda}(\theta, x)=\left(\theta+\alpha, 2+\varphi_\lambda(\theta, x)-\frac{1}{x}\right).
$$
Notice that the explicit form \eqref{e.blambda} and the fact that $Z(\theta)$ is close to the identity imply that the image of any unit vector corresponding to a point in $[1/2, 5]$ under $B_\lambda(\theta)$ is going to move monotonically as $\lambda$ changes, with speed of order $\lambda$. In the notations above this means that there exists $C>1$ such that for any $x\in [1/2, 5]$ and any $\theta \in \T$, we have
$$
C^{-1}\lambda < \varphi_\lambda < C\lambda.
$$

We will need the following statement now. %to prove Theorem \ref{t.groundstatesforquasiplus}.

\begin{prop}\label{p.graphtransform}
Consider the map $F_0:\T\times [1/2, 5]\to \T\times \mathbb{R}$ given by
$$
F_0(\theta, x)=\left(\theta +\alpha, 2-\frac{1}{x}\right).
$$
Suppose $\varphi_\lambda:\T\times [1/2, 5]\to \mathbb{R}$ is a positive smooth function such that there exists $C>1$ (independent of $\lambda$) with
\begin{equation}\label{e.varphiconditions}
C^{-1}\lambda< \varphi_\lambda<C\lambda\ \  \text{and}\ \ \ \|\varphi_\lambda\|_{C^2}\le C\lambda.
\end{equation}
Consider
$$
F_{\varphi_\lambda}:\T\times [1/2, 5]\to \T\times \mathbb{R}, F_{\varphi_\lambda}(\theta, x)=\left(\theta+\alpha, 2+\varphi_\lambda(\theta, x)-\frac{1}{x}\right).
$$

Then there exist $C^*>0$ such that for any sufficiently small $\lambda>0$ there is $\gamma=\gamma(\lambda)>0$ such that the following holds. The map $F_{\varphi_\lambda}$ has two smooth invariant sections, $\psi_{att}:\T\to [1/2, 5]$ and $\psi_{rep}:\T\to [1/2, 5]$, such that
$$
F_{\varphi_\lambda}(\theta, \psi_{att}(\theta))=(\theta+\alpha, \psi_{att}(\theta+\alpha))
$$
and
$$
F_{\varphi_\lambda}(\theta, \psi_{rep}(\theta))=(\theta+\alpha, \psi_{rep}(\theta+\alpha)),
$$
the $C^0$-distance between $\psi_{att}$ and $\psi_{rep}$ is of order $\sqrt{\lambda}$ (additionally, the graph of $\psi_{att}$ is a subset of $\T\times [1+C^*\sqrt{\lambda}, 5]$, and the graph of $\psi_{rep}$ is a subset of  $\T\times [1/2, 1-C^*\sqrt{\lambda}]$), any point $(\theta, x)$ between the curves defined by $\psi_{att}$ and $\psi_{rep}$ is attracted to the graph of $\psi_{att}$ under the iterates of $F_{\varphi_\lambda}$, and to the graph of $\psi_{rep}$ under the iterates of $F_{\varphi_\lambda}^{-1}$. In addition, for any point $(\theta, x)\in \T\times [1+C^*\sqrt{\lambda}, 5]$, we have $\left|\frac{dF_{\varphi_\lambda}}{dx}\right|<1-\gamma$, and for any point $(\theta, x)\in \T\times [1/2, 1-C^*\sqrt{\lambda}]$, we have $\left|\frac{dF_{\varphi_\lambda}}{dx}\right|>1+\gamma$.
\end{prop}
\begin{remark}
Notice that in the context of Proposition \ref{e.varphiconditions} the cylinder between the graphs of the functions $\psi_{rep}$ and $\psi_{att}$ is invariant under the map $F_{\varphi_\lambda}$, and therefore all (positive or negative) iterates of $F_{\varphi_\lambda}$ are well defined for any initial point in that cylinder.
\end{remark}

\begin{proof}[Proof of Proposition \ref{p.graphtransform}.]
For a given $\theta \in \T$, denote $\tilde F_\theta: [1/2, 5]\to \mathbb{R}$, $\tilde F_\theta(x)=2+\varphi_\lambda(\theta, x)-\frac{1}{x}$. The conditions \eqref{e.varphiconditions} imply that there are $C^{\#}>0$ and $C^{\#\#}>0$ independent of $\lambda$ such that $\tilde F_\theta$ has two fixed points, $x_{att, \theta}$ and $x_{rep, \theta}$, such that
$$
1+C^{\#}\sqrt{\lambda}<x_{att, \theta}<1+C^{\#\#}\sqrt{\lambda}
$$
and
$$
1-C^{\#\#}\sqrt{\lambda}<x_{rep, \theta}<1-C^{\#}\sqrt{\lambda}.
$$
Set $C^*=\frac{1}{10}C^{\#}$. Let us show that there exists $\gamma>0$, $\gamma=\gamma(\lambda)$ (in fact, one can take $\gamma(\lambda)=5C^*\sqrt{\lambda}$), such that
$$
\left|\frac{d}{dx}\tilde F_\theta(x)\right|<1-\gamma\ \ \text{for}\ \ x\in [1+C^*\sqrt{\lambda}, 5]
$$
and
$$
\left|\frac{d}{dx}\tilde F_\theta(x)\right|>1+\gamma\ \ \text{for}\ \ x\in [1/2, 1-C^*\sqrt{\lambda}].
$$
Indeed, we have
$$
\frac{d}{dx}\tilde F_\theta(x)=\frac{\partial \varphi_\lambda}{\partial x}(\theta, x)+\frac{1}{x^2},
$$
and if $x\in [1+C^*\sqrt{\lambda}, 5]$, we get:
\begin{align*}
\left| \frac{d}{dx} \tilde F_\theta(x) \right| & \le C \lambda + \frac{1}{(1+C^*\sqrt{\lambda})^2} \\
& < C \lambda + \frac{1}{1+2C^*\sqrt{\lambda}} \\
& < C \lambda + 1 - 4 C^*\sqrt{\lambda} \\
& < 1 - 5 C^*\sqrt{\lambda} \\
& = 1 - \gamma
\end{align*}
if $\lambda$ is sufficiently small.

Similarly, we get that for $x\in[1/2, 1-C^*\sqrt{\lambda}]$, we have
$$
 \left|\frac{d}{dx}\tilde F_\theta(x)\right|\ge 1+5C^*\sqrt{\lambda}=1+\gamma.
$$
Notice that $F_{\varphi_\lambda}$ maps $\T\times [1+C^{\#}\sqrt{\lambda}, 1+C^{\#\#}\sqrt{\lambda}]$ into itself, and contracts along the $x$-coordinate. The standard graph transform technique (see, e.g., \cite[Theorem 3.2]{HPS}) shows that
$$
\bigcap_{n\in \mathbb{N}} F^n_{\varphi_\lambda}\left(\T\times [1+C^{\#}\sqrt{\lambda}, 1+C^{\#\#}\sqrt{\lambda}]\right)
$$
is a smooth invariant curve that can be represented as the graph of a function $\psi_{att} : \T \to [1+C^{\#}\sqrt{\lambda}, 1+C^{\#\#}\sqrt{\lambda}]$. Similarly, one can construct $\psi_{rep}$ as a function with the graph
$$
\bigcap_{n\in \mathbb{N}} F^{-n}_{\varphi_\lambda}\left(\T\times [1-C^{\#\#}\sqrt{\lambda}, 1-C^{\#}\sqrt{\lambda}]\right).
$$
This completes the proof of Proposition \ref{p.graphtransform}.
\end{proof}

\begin{proof}[Proof of Theorem \ref{t.groundstatesforquasiplus}.]
The proof of Theorem \ref{t.groundstatesforquasiplus} is parallel to the proof of Theorem \ref{t.groundstates}.

Suppose $E=E^*+\lambda-a$, where $a\in (0, \lambda)$ is small, $a\ll \lambda$. Consider the maps
$$
F_{\varphi_{\lambda-a}}, F_{\varphi_{-a}}:\T\times [1/2, 5]\to \T\times \mathbb{R}^+
$$
given by
$$
F_{\varphi_{\lambda-a}}(\theta, x)=\left(\theta+\alpha, 2+\varphi_{\lambda-a}(\theta, x)-\frac{1}{x}\right)
$$
and
$$
F_{\varphi_{-a}}(\theta, x)=\left(\theta+\alpha, 2-\varphi_{-a}(\theta, x)-\frac{1}{x}\right),
$$
where
$$
C^{-1}\lambda<\varphi_{\lambda-a}<C\lambda, \|\varphi_{\lambda-a}\|_{C^2}\le C\lambda
$$
and
$$
C^{-1}a<\varphi_{-a}<Ca, \|\varphi_{-a}\|_{C^2}\le Ca.
$$
Consider the cylinder $U_{\lambda, a}\subset \T\times [1/2, 5]$ between the curve $\T\times \{1+C^*\sqrt{\lambda}\}$ and the graph of $\psi_{att}$. The next statement is an analog of Lemma \ref{l.inversesequence}.

\begin{lemma}\label{l.inversecylinder}
Given a small $\lambda>0$, for all sufficiently small $a>0$ and for any $(\theta, x)\in U_{\lambda, a}$, we have $F_{\varphi_{\lambda-a}}^{-1}(\theta, x)\in U_{\lambda, a}$ or $F_{\varphi_{-a}}^{-1}(\theta, x)\in U_{\lambda, a}$ (it is possible that both inclusions hold).
\end{lemma}

\begin{proof}%[Proof of Lemma \ref{l.inversecylinder}.]
Let us consider $F_{\varphi_{\lambda-a}}(U_{\lambda, a})$. It is a cylinder between the graph of $\psi_{att}$ and the curve $F_{\varphi_{\lambda-a}}(\T\times \{1+C^*\sqrt{\lambda}\})$. Notice that the curve $\T\times \{1+C\lambda+C^*\sqrt{\lambda}\}$ must be inside of the cylinder $F_{\varphi_{\lambda-a}}(U_{\lambda, a})$. Therefore, for any $(\theta, x)$ between the graph of $\psi_{att}$ and $\T\times \{1+C\lambda+C^*\sqrt{\lambda}\}$, we have $F^{-1}_{\varphi_{\lambda-a}}\in U_{\lambda, a}$.

Let us now consider a point $(\theta, x)\in \T\times [1+C^*\sqrt{\lambda}, 1+C\lambda+C^*\sqrt{\lambda}]$. Let us show that $F^{-1}_{-a}(\theta, x)\in U_{\lambda, a}$. Indeed, let us check first that $F^{-1}_0(\theta, x)\in U_{\lambda, a}$. We have
$$
F_0^{-1}(\theta, x)=\left(\theta-\alpha, \frac{1}{2-x}\right),
$$
so the $x$-coordinate of the point $F^{-1}_0(\theta, x)$ must be inside of the interval $\left[\frac{1}{1-C\lambda-C^*\sqrt{\lambda}}, \frac{1}{1-C^*\sqrt{\lambda}}\right]$. Notice that
$$
\frac{1}{1-C\lambda-C^*\sqrt{\lambda}}>1+C^*\sqrt{\lambda}
$$
and
$$
\frac{1}{1-C^*\sqrt{\lambda}}<1+2C^*\sqrt{\lambda}<1+C^{\#}\sqrt{\lambda},
$$
and hence
$$
\T\times \left[\frac{1}{1-C\lambda-C^*\sqrt{\lambda}}, \frac{1}{1-C^*\sqrt{\lambda}}\right]\subset U_{\lambda, a}.
$$
By continuity, if $a>0$ is small enough, for any point $(\theta, x)\in \T\times [1+C^*\sqrt{\lambda}, 1+C\lambda+C^*\sqrt{\lambda}]$, we have $F^{-1}_{\varphi_{-a}}(\theta, x)\in U_{\lambda, a}$.
\end{proof}
Let us now denote by $\hat U_{\lambda,a}$ the cylinder between the graph of $\psi_{rep}$ and the curve $\T\times \{1-C^*\sqrt{\lambda}\}$. The next statement is an analog of Lemma \ref{l.forwardsequence}, and the proof is completely parallel:

\begin{lemma}\label{l.forwardcylinder}
Given a small $\lambda>0$, for all sufficiently small $a>0$ and for any $(\theta, x)\in \hat U_{\lambda, a}$, we have $F_{\varphi_{\lambda-a}}(\theta, x)\in \hat U_{\lambda, a}$ or $F_{\varphi_{-a}}(\theta, x)\in \hat U_{\lambda, a}$ (it is possible that both inclusions hold).
\end{lemma}

%\begin{proof}[Proof of Lemma \ref{l.forwardcylinder}.]
%The proof is completely similar to the proof of Lemma \ref{l.inversecylinder}.
%\end{proof}

Now, let us state an analog of Lemma \ref{l.transit}:

\begin{lemma}\label{l.transitcyl}
For any sufficiently small $\lambda>0$, $a>0$, and any point $(\theta, x)\in  U_{\lambda, a}$, there exists $k\in \mathbb{N}$ such that $F_{\varphi_{-a}}^k(\theta, x)\in \hat U_{\lambda, a}$.
\end{lemma}

Finally, for any $(\theta, x) \in U_{\lambda, a}$, we can use Lemma~\ref{l.inversecylinder} to generate a sequence of maps $F_{\varphi_{\lambda-a}}$ and $F_{\varphi_{-a}}$ such that under the application of their inverses, the orbit of $(\theta, x)$ remains in $U_{\lambda, a}$ and the corresponding $\frac{d}{dx}$ derivative increases exponentially, and Lemmas~\ref{l.forwardcylinder} and \ref{l.transitcyl} can be used to construct a forward sequence of iterates with similar properties. This gives a solution that is an eigenfunction that corresponds to the energy $E=E^*+\lambda-a$ and has positive entries (hence is a ground state).
\end{proof}

\begin{proof}[Proof of Theorem \ref{t.main1}.]
Notice that due to Remark \ref{r.redtobernoulli} it is enough to prove Theorem \ref{t.main1} in the case of Bernoulli random potential that takes values $0$ and $\lambda$, in which case it follows from Theorem \ref{t.groundstatesforquasiplus}: the fact that for each energy in the interval in question, there exists a realization of the random potential, for which the energy is an eigenvalue implies that the Schr\"odinger cocycle at this energy is not uniformly hyperbolic. Johnson's theorem then implies that the energy must belong to the almost sure spectrum. Here we used some standard terminology and results for which we refer the reader to \cite{DF21}.
\end{proof}

%\begin{coro}
%The almost sure spectrum of the operator \eqref{e.quasiperplusBernoulli} with $\{\omega_n\}$ being a random Bernoulli sequence of $0$'s and $\lambda$'s (for any given small $\lambda>0$) contains an interval (and hence is not a Cantor set).
%\end{coro}
%
%It would be interesting to understand what happens when $\lambda$ is large. In particular, would that provide a counterexample to the following conjecture?
%
%\begin{conj}
%For any ergodic background potential, adding a Bernoulli potential (or any random bounded iid potential) leads to an almost sure spectrum which contains an interval. Or, even stronger, is a finite union of intervals.
%\end{conj}

\begin{appendix}

\section{A Remark on Unbounded Background Potentials}\label{app.Unbounded}

In this section we expand on a comment made in Remark~\ref{r.problem1.2}. There it was pointed out that if the boundedness of $V_\mathrm{bg}$ is not assumed, it is possible to find a counterexample to the problem posed in Problem~\ref{prob.2}, namely one can find an unbounded background potential $V_\mathrm{bg}$ and a compactly supported non-degenerate $\nu$ such that all the resulting operators $H_\omega$ have empty essential spectrum, and hence it is impossible for the spectrum to contain an interval. In fact, the underlying principle is both simple and purely deterministic:

\begin{prop}\label{p.noessspect}
Suppose $V_\mathrm{bg} : \Z \to \R$ is given by $V_\mathrm{bg}(n) = n$, $n \in \Z$,
%. Then for every compactly supported single-site distribution $\nu$ and every $\omega \in (\mathrm{supp} \, \nu)^\Z$, the operator \eqref{e.oper} with potential \eqref{e.potentialsum} has empty essential spectrum.
and $V_\mathrm{b} : \Z \to \R$ is bounded. Then the essential spectrum of the Schr\"odinger operator $H$ in $\ell^2(\Z)$ with potential $V_\mathrm{bg}+ V_\mathrm{b}$ is empty.
\end{prop}

\begin{proof}
Consider the transfer matrices
$$
\Pi_n=\begin{pmatrix}
        E-n-V_\mathrm{b}(n) & -1 \\
        1 & 0 \\
      \end{pmatrix}
$$
and the cone in $\mathbb{R}^2$ given by
$$
K=\left\{\bar v=(v_1, v_2)\ |\ |v_1|>|v_2|\right\}.
$$

Set $M := \|V_\mathrm{b}\|_\infty$ and fix a bounded open interval $I\subset \mathbb{R}$. Then for any $E\in I$, all sufficiently large $n$, for any $\bar v\in K$ we have
\begin{align*}
|(E-n-V_\mathrm{b}(n))v_1-v_2 |  & \ge |n-E-M| \cdot |v_1|-|v_2| \\
& \ge |n-E-M| \cdot |v_1| - |v_1| \\
& \ge  (n-E-M-1)|v_1|>|v_1|,
\end{align*}
Thus, since
$$
\Pi_n\bar v = \begin{bmatrix}(E-n-V_\mathrm{b}(n))v_1-v_2 \\ v_1 \end{bmatrix},
$$
$\Pi_n$ sends the cone $K$ to itself. Moreover, $\Pi_n$ expands the vectors in $K$. Indeed,
\begin{align*}
|(E-n-V_\mathrm{b}(n))v_1-v_2 | + |v_1| & \ge (n-E-M)|v_1| \\
& > \frac{1}{2}(n-E-M)(|v_1|+|v_2|).
\end{align*}

It follows that, for sufficiently large $N$ (with a largeness condition that depends on $M$ and $I$), the restriction of $H$ to the half line $[N,+\infty)$ with a Dirichlet boundary condition at $N$ has no spectrum on the open interval $I$, as the solution obeying the boundary condition is not a generalized eigenfunction for any $E \in I$.\footnote{Here we used one of the well-known aspects of Shnol's theorem: the spectrum is given by the closure of the set of energies for which there exists a non-trivial polynomially bounded solution satisfying the boundary condition. Since we are dealing with an unbounded potential, let us mention that Shnol's theorem holds in this setting as well \cite{H19}.}

This in turn shows that the restriction of the operator $H$ to the half-line $[0, +\infty)$ (with any self-adjoint boundary condition at zero) has no essential spectrum on the interval $I$.\footnote{This also follows quickly from known results: first, the variation of the boundary condition at zero falls within rank-one perturbation theory \cite{S95} and invariance of the essential spectrum is clear; second, the change of the left endpoint of the half-line can be investigated via the standard coefficient stripping technique \cite{S11} and invariance of the essential spectrum is then again clear; third, the variation of the potential on the finite inserted piece leaves the essential spectrum invariant due to Weyl's theorem.}

Similar arguments show that the half-line operator obtained by restriction of $H$ to $(-\infty, 0]$ (again with any self-adjoint boundary condition at zero) has no essential spectrum on the interval $I$.

Combining the two statements, it then follows that $H$ itself has no essential spectrum on the interval $I$. Since the choice of the bounded open interval $I$ was arbitrary, Proposition~\ref{p.noessspect} follows.
\end{proof}

\section{The Almost Sure (Essential) Spectrum}\label{sec.almostsurespectrum}

In this section we discuss non-randomness aspects of spectra that are well known in the ergodic setting. However, since we are interested in non-stationary random potentials in this paper, we need extensions of these results. It will turn out that Kolmogorov's zero-one law can serve as a substitute, leading to the non-randomness of the \emph{essential} spectrum. This will be explained in Subsection~\ref{ssec.kolmogorov} below. In Subsection~\ref{ssec.kotsupp} we then derive an extension of Kotani's support theorem from the ergodic setting to the non-stationary case. Along the way we also explain why under suitable additional assumptions, one does have a non-random spectrum, as formulated in Theorem~\ref{t.almostsurespectrum} in the Introduction. Before presenting the proofs of these results in Subsections~\ref{ssec.kolmogorov} and \ref{ssec.kotsupp} we recall a useful characterization of the essential spectrum of a deterministic Schr\"odinger operator in terms of transfer matrix behavior in Subsection~\ref{ssec.essspechcar}, as this tool will be used in those proofs.

\subsection{A Characterization of the Essential Spectrum}\label{ssec.essspechcar}

The following characterization of the essential spectrum in terms of transfer matrices can be extracted from the denseness of generalized eigenvalues (energies for which there are polynomially bounded solutions of the Schr\"odinger equation) and the classical Weyl criterion; see \cite[Proposition B.2]{GK}.

\begin{prop}\label{p.crit}
Let $V:\mathbb{Z}\to \mathbb{R}$ be a bounded potential of the discrete Schr\"odinger operator $H$ acting on $\ell^2(\Z)$ via
\begin{equation}
[H u](n) = u(n+1) + u(n-1) + V(n) u(n).
\end{equation}
Then
%\vspace{5pt}
%
%1) Energy $E\in \mathbb{R}$ belongs to the spectrum of the operator $H$ if and only if there exists $K>0$ such that for any $N\in \mathbb{N}$ there is $m\in \mathbb{Z}$ and a unit vector $\bar u$, $|\bar u|=1$,  such that $|T_{[m, m+i], E}\,\bar u|\le K$ for all $|i|\le N$,
%\vspace{5pt}
energy $E\in \mathbb{R}$ belongs to the essential spectrum of the operator $H$ if and only if there exists $K>0$ such that for any $N\in \mathbb{N}$ there is  a sequence $\{m_j\}_{j\in \mathbb{N}}, m_j\in \mathbb{Z},$ with $|m_j-m_{j'}|>2N$ if $j\ne j'$, and unit vectors $\bar u_j$, $|\bar u_j|=1$,  such that $|T_{[m_j, m_j+i], E}\,\bar u_j|\le K$ for all $|i|\le N$ and all $j\in \mathbb{N}$,  where $T_{[m, m+i], E}$ is the product of transfer matrices given by
$$
T_{[m, m+i], E}=\left\{
               \begin{array}{lll}
                 \Pi_{m+i-1, E}\ldots \Pi_{m, E}, & \hbox{if $i> 0$;} \\
                 \text{\rm Id},  & \hbox{if $i= 0$;} \\
                 \Pi_{m+i, E}^{-1}\ldots \Pi_{m-1, E}^{-1}, & \hbox{if $i<0$,}
               \end{array}
             \right.
$$
and $\Pi_{n, E}=\left(
               \begin{array}{cc}
                 E-V(n) & -1 \\
                 1 & 0 \\
               \end{array}
             \right)$.

\end{prop}

\subsection{Existence of the Almost Sure (Essential) Spectrum}\label{ssec.kolmogorov}

%In this section we discuss the existence of a non-random set that can serve as the spectrum, or the essential spectrum (depending on the setting), of the given random operator for almost all realizations of the random potential.

%\begin{theorem}\label{t.almostsurespectrum}
%Suppose $T : X \to X$ is a minimal homeomorphism of a compact metric space $X$ and $f \in C(X,\R)$. Suppose further that $\nu$ is a compactly supported single-site distribution. Then there is a compact $\Sigma_\nu \subset \R$ such that for every $x \in X$ and $\mu = \nu^\Z$-almost every $\omega$, we have $\sigma(H_\omega) = \Sigma_\nu$, where $H_\omega$ is given by \eqref{e.oper} and \eqref{e.potentialsum} with $V_\mathrm{bg}(n) = f(T^n x)$.
%\end{theorem}

Recall that Theorem \ref{t.almostsurespectrum} asserts the existence of a non-random spectrum under the assumption that the background potential is generated by continuous sampling along the orbit of a minimal homeomorphism.

\begin{proof}[Proof of Theorem \ref{t.almostsurespectrum}]
Denote by $G:(\text{supp}\,\nu)^{\mathbb{Z}}\to (\text{supp}\,\nu)^{\mathbb{Z}}$ the left shift on the space of sequences $\omega\in (\text{supp}\,\nu)^{\mathbb{Z}}$, and consider the map
$$
T\times G:X\times (\text{supp}\,\nu)^{\mathbb{Z}}\to X\times (\text{supp}\,\nu)^{\mathbb{Z}}.
$$

We have the following statement:

\begin{lemma}\label{l.minmix}
For every $x\in X$ and $\nu^\mathbb{Z}$-almost every $\omega$, the $(T\times G)$-orbit of $(x, \omega)$ is dense in $X\times (\text{supp}\,\nu)^{\mathbb{Z}}$.
\end{lemma}

\begin{proof}%[Proof of Lemma \ref{l.minmix}]
Fix any small $\varepsilon>0$, and any point $(y, \omega')\in X\times (\text{supp}\, \nu)^\mathbb{Z}$. Due to the minimality of $T:X\to X$, there exists a sequence $m_j\in \mathbb{N}$, $m_j\to \infty$ as $j\to \infty$, such that $\text{dist}_X(T^{m_j}(x), y)<\varepsilon$. For any $j\in \mathbb{N}$, the probability of the event
$$
\left\{|\tilde \omega_{m_j+i}-\omega'_{m_j+i}|<\varepsilon\ \text{for any}\ |i|<\frac{1}{\varepsilon}, \ \text{where}\ (T\times G)^{m_j}(x, \omega)=(T^{m_j}(x), \tilde \omega)\right\}
$$
is bounded away from zero, and if $|m_j-m_{j'}|>\frac{2}{\varepsilon}$, then those events are independent. Hence with probability one infinitely many of them must happen. Since $\varepsilon>0$ could be chosen arbitrarily small, Lemma \ref{l.minmix} follows.
\end{proof}
Since for any two potentials defined by initial conditions that have dense orbits, the spectra of the corresponding operators coincide, Theorem \ref{t.almostsurespectrum} follows.
\end{proof}

In cases where the background potential is not of the form considered in Theorem \ref{t.almostsurespectrum}, we have the following substitute result. We are grateful to Victor Kleptsyn for communicating it to us.

\begin{theorem}\label{t.essspect}
Suppose $\{\nu_n\}_{n\in \mathbb{Z}}$ is a family of probability distributions on $\mathbb{R}$ with uniformly bounded supports. Let $V:\mathbb{Z}\to \mathbb{R}$ be a random potential chosen (independently at each site) with respect to the measure $\mu=\prod_{n \in \Z} \nu_n$. Then there exists a (non-random) compact set $\Sigma\subset \mathbb{R}$ such that $\mu$-almost surely, the essential spectrum of the discrete Schr\"odinger operator with potential $V$ is equal to $\Sigma$.
\end{theorem}

\begin{proof}%[Proof of Theorem \ref{t.essspect}.]
Notice that Proposition \ref{p.crit} implies that for a given point $E_0\in \mathbb{R}$, the event ``the energy $E_0$ belongs to the essential spectrum of the Schr\"odinger operator with the random potential $V$'' is a ``tail event'', and hence due to Kolmogorov's zero-one law must have probability either zero or one. Similarly, for any given closed interval $I$, the event ``$I$ has non-empty intersection with the essential spectrum'' is a ``tail event'', and has probability either zero or one. Since there is a countable base of the topology of $\mathbb{R}$ consisting of intervals, and a countable intersection of tail events is a tail event, Theorem~\ref{t.essspect} follows.
\end{proof}

\subsection{Monotonicity: The Support Theorem}\label{ssec.kotsupp}

Here we formulate and prove a generalization of Kotani's support theorem, originally proved in the continuum ergodic setting in \cite{K85}:

\begin{theorem}\label{t.supportthm}
Let $\nu_1$ and $\nu_2$ be two probability distributions on $\mathbb{R}$ with bounded support.  Let $V_{\mathrm{bg}}:\mathbb{Z}\to \mathbb{R}$ be a bounded background potential. Denote by $\Sigma_1$ the almost sure essential spectrum of the discrete Schr\"odinger operator given by the random potential $V_{\mathrm{bg}}+V^\mathrm{AM}_{\omega}$, where $V^\mathrm{AM}_\omega$ is a random sequence generated with respect to the distribution $\nu_1$ at each site. Define $\Sigma_2$ similarly, using the distribution $\nu_2$. If $\mathrm{supp} \, \nu_1 \subseteq \mathrm{supp}\, \nu_2$, then $\Sigma_1 \subseteq \Sigma_2$.
\end{theorem}

\begin{proof}%[Proof of Theorem \ref{t.supportthm}.]
Suppose that $E_0\in \Sigma_1$. Then, due to Proposition \ref{p.crit} there exists $K>0$ such that for any $N\in \mathbb{N}$, there are a sequence $\{m_j\}_{j\in \mathbb{N}}, m_j\in \mathbb{Z},$ with $|m_j-m_{j'}|>2N$ if $j\ne j'$, unit vectors $\bar u_j$, $|\bar u_j|=1$, and $(2N+1)$-tuples $\{t^1_{-N, j}, \ldots, t^1_{0, j}, \ldots, t^1_{N, j}\}$ of real numbers, $t^1_{i,j}\in \text{supp}\, \nu_1$, such that $|T_{[m_j, m_j+i], E}\,\bar u_j|\le K$ for all $|i|\le N$ and all $j\in \mathbb{N}$, where $T_{[m_j, m_j+i]}$ are the products of transfer matrices
$$
\Pi_{m_j+i, E_0}=\left(
               \begin{array}{cc}
                 E_0-(V_{\mathrm{bg}}(m_j+i)+ t^1_{i,j}) & -1 \\
                 1 & 0 \\
               \end{array}
             \right), \ \ |i|\le N.
$$
By continuity it follows that there exists $\varepsilon_{N}>0$ such that for any $2N+1$-tuple $\{\tilde t_{-N, j}, \ldots, \tilde t_{0, j}, \ldots, \tilde t_{N, j}\}$ with $|t^1_{i,j}-\tilde t_{i,j}|<\varepsilon_N$ we have $|\tilde T_{[m_j, m_j+i], E}\,\bar u_j|\le 2K$, where $\tilde T_{[m_j, m_j+i]}$ are the products of transfer matrices
$$
\tilde \Pi_{m_j+i, E_0}=\left(
               \begin{array}{cc}
                 E_0-(V_{\mathrm{bg}}(m_j+i)+ \tilde t_{i,j}) & -1 \\
                 1 & 0 \\
               \end{array}
             \right), \ \ |i|\le N.
$$
Since $\text{supp}\, \nu_1\subseteq \text{supp}\, \nu_2$, this implies that with positive (and, by compactness arguments, uniformly in $j\in \mathbb{N}$ bounded away from zero) probability, a random sequence generated by i.i.d.\ random variables distributed with respect to $\nu_2$ will coincide (up to an error not greater than $\varepsilon_N>0$) with the $(2N+1)$-tuple $\{t^1_{-N, j}, \ldots, t^1_{0, j}, \ldots, t^1_{N, j}\}$ over the interval of indices $[m_j-N, m_j+N]$. Due to the second Borel–Cantelli Lemma, if the sum of probabilities of a sequence of independent events is infinite, with probability one infinitely many of those events happens. Therefore,  another application of Proposition~\ref{p.crit} implies that $\Sigma_1 \subseteq \Sigma_2$.
\end{proof}

\begin{proof}[Proof of Theorem \ref{t.supporttheorem}.]
Since in the ergodic case the almost sure spectrum coincides with the almost sure essential spectrum, Theorem \ref{t.supporttheorem} follows from Theorem \ref{t.supportthm}.
\end{proof}

\section{Generalized Ground States}\label{sec.basicsofGS}

\subsection{The Deterministic Setting}

Let us consider a Schr\"odinger operator
\begin{equation}\label{e.operapp}
[H \psi](n) = \psi(n+1) + \psi(n-1) + V(n) \psi(n)
\end{equation}
in $\ell^2(\Z)$ with a bounded potential $V : \Z \to \R$. Thus, $H$ is a bounded self-adjoint operator and its spectrum $\sigma(H)$ is a compact subset of $\R$. We refer to
\begin{equation}\label{e.gseapp}
E_\mathrm{max} := \max \sigma(H)
\end{equation}
as the \emph{ground state energy}. Traditionally, one considers $-\Delta + V$ and refers to the bottom of the spectrum as the ground state energy. Since it is customary to drop the minus sign when considering discrete Schr\"odinger operators, one then switches focus from the bottom to the top of the spectrum.

Our goal is to discuss the \emph{ground state}. In informal terms, this is ``the smallest'' solution of the difference equation
\begin{equation}\label{e.eveapp}
u(n+1) + u(n-1) + V(n) u(n) = E u(n)
\end{equation}
for $E = E_\mathrm{max}$. In the traditional setting, when considering atomic models $-\Delta+V$ with $V$ having a non-trivial negative part and $|V(x)| \to 0$ as $|x| \to \infty$, the bottom of the spectrum (usually) is a discrete eigenvalue and one is often able to show that it is simple. A normalized associated eigenfunction, which could typically be shown to be strictly positive, was then referred to as the ground state. In our general setting, there may not be a square-summable eigenfunction, and hence we will need a more general concept to identify a ground state.

Let us recall the famous definition of subordinacy due to Gilbert and Pearson \cite{GP87}.

\begin{definition}
A solution $u$ of \eqref{e.eveapp} is called \emph{subordinate at} $+ \infty$ if it does not vanish identically and we have
$$
\lim_{N \to \infty} \frac{\sum_{n=1}^N |u(n)|^2}{\sum_{n=1}^N |\tilde u(n)|^2} = 0
$$
for every solution $\tilde u$ of \eqref{e.eveapp} that is linearly independent from $u$ (i.e., $\tilde u$ is not a multiple of $u$).
\end{definition}

\begin{remark}\label{r.subordinacy}
(a) Subordinacy at $- \infty$ is defined analogously.
\\[1mm]
(b) By the constancy of the Wronskian (or the equivalent fact that the transfer matrices are unimodular), we cannot have two linearly independent square-summable solutions. Thus, every solution that is square-summable at $+ \infty$ is subordinate at $+ \infty$, and hence subordinacy generalizes square-summability.
\\[1mm]
(c) If there is a solution of \eqref{e.eveapp} that is subordinate at $\pm \infty$, then it is unique up to a multiplicative constant. In this sense one can say that it is ``the smallest solution.''
\\[1mm]
(d) We will eventually be interested in dynamically defined potentials, where the associated Schr\"odinger cocycle at energy $E_\mathrm{max}$ is reducible to a constant parabolic matrix. Thus, for each element of the hull, there will then be a unique (up to a constant multiple) solution $u_\mathrm{bdd}^\pm$ that is bounded at $\pm \infty$, while each linearly independent solution grows linearly. It is easy to check that in this scenario, $u_\mathrm{bdd}^\pm$ is subordinate at $\pm \infty$.
\end{remark}

\begin{theorem}\label{t.positivegroundstate}
Suppose that $E = E_\mathrm{max}$ and $u^\pm$ is a solution of \eqref{e.eveapp} that is subordinate at $\pm \infty$. Then, up to a multiplicative constant, we have $u^\pm(n) > 0$ for every $n \in \Z$.
\end{theorem}

\begin{proof}
We only consider the $+$ case, as the proof for the $-$ case is completely analogous.

Let us consider the ratios
$$
d(n) = \frac{u^+(n)}{u^+(n-1)}.
$$
As $u^+$ is subordinate, $u^+$ does not vanish identically and hence cannot have two consecutive zeros. Thus, while it can happen that $u^+(n-1) = 0$, we then must have $u^+(n) \not= 0$, and we can unambiguously set $d(n) := \infty$ in this case. In all other cases, $d(n)$ is a finite (complex) number.

In fact, we have $d(n) \in \R \cup \{\infty\}$ for every $n \in \Z$. To see this, assume this fails. Then a conjugate solution $\bar u^+$ is linearly independent from $u^+$, solves \eqref{e.eveapp} as well (because $E$ is real), and is subordinate at $+ \infty$ as well (for obvious reasons); a contradiction.

Note next that the claim of the theorem is equivalent to
$$
d(n) \in (0,\infty) \quad \text{for every } n \in \Z.
$$
Moreover, if $d(n) = \infty$ for some $n \in \Z$, then $d(n-1) = 0$. Thus, the failure of the claim of the theorem is equivalent to
\begin{equation}\label{e.failure}
\exists k \in \Z : d(k) \in (-\infty,0].
\end{equation}

Assume that \eqref{e.failure} holds. We consider the restriction $H_k^+$ of $H$ to $\ell^2(\{ k, k+1, k+2, \ldots \})$ with Dirichlet boundary condition. We consider the Weyl-Titchmarsh function
\begin{equation}\label{e.WTfunction}
m_k^+(z) := \langle \delta_k, (H^+_k - z)^{-1} \delta_k \rangle = \int \frac{d\mu_k^+(E)}{E - z},
\end{equation}
where $\mu_k^+$ is the spectral measure corresponding to the pair $(H_k^+,\delta_k)$ and $z \in \C \setminus \sigma(H^+_k)$. It follows from general results (see, e.g., \cite{DF21}) that we can also write
\begin{equation}\label{e.WTfunction2}
m_k^+(z) = - \frac{u^+_z(k)}{u^+_z(k-1)},
\end{equation}
where $u^+_z$ is subordinate at $+\infty$ and solves \eqref{e.eveapp} with $E = z$.

By the min-max principle, we have that $\sigma(H^+_k) \subset (-\infty,E_\mathrm{max}]$ and hence $\mu^+_k$ gives no weight to $(E_\mathrm{max},\infty)$. It therefore follows from \eqref{e.WTfunction} that
\begin{equation}\label{e.WTfunctionconsequence1}
z \in (E_\mathrm{max},\infty) \quad \Rightarrow \quad m_k^+(z) < 0.
\end{equation}
Moreover, it also follows from \eqref{e.WTfunction} that
\begin{equation}\label{e.WTfunctionconsequence2}
z \in (E_\mathrm{max},\infty) \quad \Rightarrow \quad (m_k^+)'(z) = \int \frac{d\mu_k^+(E)}{(E - z)^2} > 0.
\end{equation}
Obviously, \eqref{e.WTfunctionconsequence1} and \eqref{e.WTfunctionconsequence2} show that the following limit exists:
\begin{equation}\label{e.WTfunctionlimit}
m_k^+(E_\mathrm{max}) = \lim_{z \downarrow E_\mathrm{max}} m^+_k(z) \in [-\infty, 0).
\end{equation}
On the other hand, the relation \eqref{e.WTfunction2} then extends to $z = E_\mathrm{max}$ by subordinacy theory (see, e.g., \cite{DF21}), and we obtain
\begin{equation}\label{e.m+kandd}
m_k^+(E_\mathrm{max}) = - d(k) \in [0,\infty).
\end{equation}
Since \eqref{e.WTfunctionlimit} and \eqref{e.m+kandd} are incompatible, it follows that \eqref{e.failure} is impossible, and the proof is finished.
\end{proof}

It is of interest to find sufficient conditions for the assumption of Theorem~\ref{t.positivegroundstate}. Let us discuss the presence of a subordinate solution on the right half line at the top of the spectrum. The case of the left half line is of course similar. In applications we will need the coincidence of the energy, that is, the top of the spectrum will have to be the same for the whole line, the left half line, and the right half line. This will be the case, for example, in the Anderson model.

So, we consider an operator $H$ in $\ell^2(\N)$, $\N = \{ 1, 2, 3, \ldots \}$, acting as \eqref{e.oper} together with a Dirichlet boundary condition at the origin,
\begin{equation}\label{e.dbc}
\psi(0) = 0.
\end{equation}
(Note that in the proof above, this operator would have been denoted by $H_1^+$.)

We have the following result. %, which we state without proof since it is not needed in the main body of the paper.

\begin{prop}\label{prop.HLexistenceofSS}
Consider the setting just described and define the ground state energy $E_\mathrm{max}$ as in \eqref{e.gseapp}. Then the difference equation \eqref{e.eveapp} admits a subordinate solution at $\infty$ for $E = E_\mathrm{max}$.
\end{prop}

\begin{proof}%[Proof of Proposition \ref{prop.HLexistenceofSS}.]
%The proof of this proposition uses some arguments from the proof of Theorem~\ref{t.positivegroundstate}.\marginpar{This perhaps calls for some reorganization of the text.}
%
Denote the spectral measure of $H$ and $\delta_1$ by $\mu$. Thus, the associated Weyl-Titchmarsh function $m : \C_+ \to \C_+$ is given by
\begin{equation}\label{e.WTfunctionHL}
m(z) := \langle \delta_1, (H - z)^{-1} \delta_1 \rangle = \int \frac{d\mu_(E)}{E - z}.
\end{equation}
Obviously, the right-hand side makes sense for every $z \not\in \mathrm{supp} \, \mu = \sigma(H)$, and we will make use of this fact.

It follows from general results (see, e.g., \cite{DF21}) that we can also write
\begin{equation}\label{e.WTfunction2HL}
m(z) = - \frac{u^+_z(1)}{u^+_z(0)},
\end{equation}
where $u^+_z$ is subordinate at $+\infty$ and solves \eqref{e.eveapp} with $E = z$. The existence of such a solution for every $z \not\in \mathrm{supp} \, \mu = \sigma(H)$ is clear (simply apply $(H-z)^{-1}$ to $\delta_1$, obtain a solution away from one and modify around one to turn it into a genuine solution).

By the same reasoning as in the previous proof we again have the negativity statement
\begin{equation}\label{e.WTfunctionconsequence1HL}
z \in (E_\mathrm{max},\infty) \quad \Rightarrow \quad m(z) < 0
\end{equation}
and the monotonicity statement
\begin{equation}\label{e.WTfunctionconsequence2HL}
z \in (E_\mathrm{max},\infty) \quad \Rightarrow \quad m'(z) = \int \frac{d\mu(E)}{(E - z)^2} > 0,
\end{equation}
which combined imply the limiting statement
\begin{equation}\label{e.WTfunctionlimitHL}
m(E_\mathrm{max}) := \lim_{\varepsilon \downarrow 0} m(E_\mathrm{max} + \varepsilon) \in [-\infty, 0).
\end{equation}

Observe that we have
\begin{equation}\label{e.WTfunctionlimitHL2}
\lim_{\varepsilon \downarrow 0} m(E_\mathrm{max} + i \varepsilon) = \lim_{\varepsilon \downarrow 0} m(E_\mathrm{max} + \varepsilon),
\end{equation}
that is, we claim that the limit on the left-hand side exists and equals the limit on the right-hand side. To see this, one can separate the two cases $m(E_\mathrm{max}) = - \infty$ and $m(E_\mathrm{max}) \in (-\infty,0)$ and use monotone convergence in the first case and dominated convergence in the second case to verify existence of the limit on the left-hand side, as well as equality with $m(E_\mathrm{max})$.

Thus, \eqref{e.WTfunctionlimitHL} and \eqref{e.WTfunctionlimitHL2} yield
$$
\lim_{\varepsilon \downarrow 0} m(E_\mathrm{max} + i \varepsilon) = m(E_\mathrm{max}) \in [-\infty, 0),
$$
which in turn implies the desired statement via subordinacy theory; compare \cite{DF21, JL99}.
\end{proof}

\begin{remark}
An alternative approach is presented in \cite[Section~2.3]{T00}. The idea is to first show that for $E \ge E_\mathrm{max}$, every solution of \eqref{e.eveapp} can change sign at most once, and to then consider the Dirichlet solution, which has a zero (and hence a sign change) at $n_0$, normalize it, and send $n_0$ to infinity. This recovers a positive solution $u_+(\cdot,E)$, which can be shown to be minimal among all positive solutions in a suitable sense (namely, it has values $1$ and $\phi_+(E)$ at the points $0$ and $1$, and a solution $u(\cdot,E)$ with values $1$ and $\phi(E)$ at the points $0$ and $1$ will be positive if and only if $\phi(E) \ge \phi_+(E)$). Moreover, it turns out to be strongly (i.e., no need for an average) subordinate:
$$
\lim_{n \to \infty} \frac{u_+(n,E)}{u(n,E)} = 0.
$$
A similar treatment can be performed near $-\infty$, and hence one obtains positive solutions $u_\pm(\cdot, E)$ that may or may not be linearly dependent, but which are minimal on their respective half line. This leads to the two cases where $u_\pm$ are linearly dependent (called \emph{critical}) or not (called \emph{subcritical}).

Conversely, it can be shown that the presence of a positive solution at energy $E$ implies that $\sigma(H) \subset (-\infty,E]$, so that combining the two statements, one finds that $\sigma(H) \subset (-\infty,E]$ if and only if there is a positive solution at $E$.
\end{remark}

\subsection{The Dynamically Defined Setting}

Let us take Remark~\ref{r.subordinacy}.(d) further and deduce some consequences of Theorem~\ref{t.positivegroundstate} in the case of dynamically defined potentials. Before stating them, let us describe the framework. Suppose $(\Omega,T)$ is a topological dynamical system given by a homeomorphism $T : \Omega \to \Omega$ of a compact metric space. Fix an ergodic Borel probability measure $\mu$ and a continuous sampling function $f : \Omega \to \R$. We obtain a family of potentials $\{ V_\omega \}_{\omega \in \Omega}$ given by
$$
V_\omega(n) = f(T^n \omega), \quad \omega \in \Omega, \; n \in \Z,
$$
and a family of Schr\"odinger operators $\{ H_\omega \}_{\omega \in \Omega}$ in $\ell^2(\Z)$, acting via
$$
[H_\omega \psi](n) = \psi(n+1) + \psi(n-1) + V_\omega(n) \psi(n), \quad \omega \in \Omega, \; n \in \Z.
$$
It is a fundamental result (see, e.g., \cite{DF21}) that there exists a compact set $\Sigma \subset \R$ such that $\sigma(H_\omega) = \Sigma$ for $\mu$-almost every $\omega \in \Omega$. Moreover, in case $T$ is minimal (i.e., all of its orbits are dense), then we even have $\sigma(H_\omega) = \Sigma$ for every $\omega \in \Omega$ (and the choice of the ergodic measure plays no role). Let us set
$$
E_\mathrm{max} := \max \Sigma.
$$

The difference equation associated with $H_\omega$,
\begin{equation}\label{e.ddeve}
u(n+1) + u(n-1) + V_\omega(n) u(n) = E u(n),
\end{equation}
can be recast in matrix-vector form as
$$
\begin{pmatrix} u(n+1) \\ u(n) \end{pmatrix} = \begin{pmatrix} E - V_\omega(n) & - 1 \\ 1 & 0 \end{pmatrix} \begin{pmatrix} u(n) \\ u(n-1) \end{pmatrix}.
$$
Iterating this, one sees that the resulting matrix product is generated by considering the second component of the iterates of the following skew-product
$$
(T,A_E) : \Omega \times \R^2 \to \Omega \times \R^2, \quad (\omega, v) \mapsto (T \omega, A_E(\omega) v),
$$
where
$$
A_E : \Omega \to \mathrm{SL}(2,\R), \quad \omega \mapsto \begin{pmatrix} E - f(\omega) & -1 \\ 1 & 0 \end{pmatrix}.
$$
For $n \in \Z$ we define $A^n_E : \Omega \to \mathrm{SL}(2,\R)$ by $(T,A_E)^n = (T^n,A^n_E)$ and then note that this is precisely the matrix product that sends $(u(0), u(-1))^t$ to $(u(n), u(n-1))^t$ for solutions of \eqref{e.ddeve}.

\begin{coro}\label{c.dyndefsetting}
Suppose there are $c \in \R$ and a continuous map $W: \Omega \to \mathrm{SL}(2,\R)$ such that
\begin{equation}\label{e.conjugacy}
W(T \omega)^{-1} A_{E_\mathrm{max}}(\omega) W(\omega) = A_* := \begin{pmatrix} 1 & c \\ 0 & 1 \end{pmatrix}.
\end{equation}
Then $c\ne 0$, and there are a compact cone $C$ in the open first quadrant of $\R^2$ and a continuous section
$b : \Omega \to C \setminus \{ 0 \}$ so that
\begin{itemize}

\item[{\rm (i)}] $b$ is projectively invariant under the dynamics: $[b(T\omega)] = [A_{E_\mathrm{max}}(\omega) b(\omega)]$, where $[ \cdot ] : \R^2 \setminus \{0\} \to \R \mathbb{P}^1$ denotes the canonical projection,

\item[{\rm (ii)}] for every $\omega \in \Omega$, the sequence $(A^n_{E_\mathrm{max}}(\omega) b(\omega))_{n \in \Z}$ is bounded.

\end{itemize}
\end{coro}

%\begin{remark}
%The assumption of Corollary~\ref{c.dyndefsetting} is satisfied for small analytic quasi-periodic potentials with Diophantine frequency. This was shown by Hadj-Amor in \cite{HA09}; her result is the discrete analog of Eliasson's result in the continuum case \cite{E92}. Concretely this means that $\Omega = \T^d$ for some $d \in \N$, $T : \T^d \to \T^d$ is given by $T \omega = \omega + \alpha$ where $\alpha$ satisfies a typical Diophantine condition, and the sampling function is of the form $f = \lambda g$ with $g : \T^d \to \R$ analytic and $0 < \lambda < \lambda_0(\alpha,g)$. The resulting Schr\"odinger cocycle is then shown to be reducible, provided that its rotation number is either rational or Diophantine with respect to $\alpha$. This result covers the case of interest to us, which is the case of zero rotation number as we are interested in the top of the spectrum.
%\end{remark}

\begin{proof}
%\begin{remark}\label{r.coincidenceofspectrumtops}
Recall that the discussion preceding Proposition~\ref{prop.HLexistenceofSS} points out the relevance of the coincidence of the top of the spectrum for the line and half line operators. Let us explain that this will always hold in the dynamically defined situation.

We use the natural notation $H_\omega, H^+_\omega, H^-_\omega$ for the operators in question. First of all, it is well known and not hard to see (cf., e.g., \cite{DF21}) that the essential spectra coincide almost surely, that is,
\begin{equation}\label{e.coincidenceofspectrumtops1}
\sigma_\mathrm{ess}(H_\omega) = \sigma_\mathrm{ess}(H^+_\omega) = \sigma_\mathrm{ess}(H^-_\omega) \quad \text{for $\mu$-almost every } \omega \in \Omega.
\end{equation}
Secondly, the spectrum of the whole-line operator is purely essential, that is,
\begin{equation}\label{e.coincidenceofspectrumtops2}
\sigma(H_\omega) = \sigma_\mathrm{ess}(H_\omega) \quad \text{for $\mu$-almost every } \omega \in \Omega.
\end{equation}
Finally, the min-max theorem implies as before that
\begin{equation}\label{e.coincidenceofspectrumtops3}
\max \sigma(H^\pm_\omega) \le \max \sigma(H_\omega) \quad \text{for every } \omega \in \Omega.
\end{equation}
Combining \eqref{e.coincidenceofspectrumtops1}--\eqref{e.coincidenceofspectrumtops3}, we find that
\begin{equation}\label{e.coincidenceofspectrumtops4}
\max \sigma(H^\pm_\omega) = \max \sigma(H_\omega) \quad \text{for $\mu$-almost every } \omega \in \Omega.
\end{equation}
In cases where the spectrum is $\omega$-independent, this identity then trivially extends to all $\omega$'s.

It follows from Proposition~\ref{prop.HLexistenceofSS} and (\ref{e.coincidenceofspectrumtops4}) %Remark~\ref{r.coincidenceofspectrumtops}
that whenever we have a conjugacy of the form \eqref{e.conjugacy} with a continuous map $W : \Omega \to \mathrm{SL}(2,\R)$ and \emph{some} $c \in \R$, then %the assumption of Corollary~\ref{c.dyndefsetting} is satisfied, that is,
we must have $c \not= 0$.

Now let us set
$$
\tilde b(\omega) := W(\omega) \begin{pmatrix} 1 \\ 0 \end{pmatrix} = \begin{pmatrix} W_{11}(\omega) \\ W_{21}(\omega) \end{pmatrix}.
$$
The conjugacy \eqref{e.conjugacy} shows that $\tilde b$ has property (i). Since
$$
A_*^n \begin{pmatrix} 1 \\ 0\end{pmatrix}
$$
is bounded as $n$ ranges over $\Z$, the conjugacy \eqref{e.conjugacy} also shows that
$$
A^n_{E_\mathrm{max}}(\omega) \tilde b(\omega)
$$
is bounded as $n$ ranges over $\Z$. Thus, $\tilde b$ has property (ii) as well. Moreover, it takes values in $\R^2 \setminus \{ 0 \}$. It remains to show that it can be modified so that (the two properties are preserved and) it takes values in a compact cone in the open first quadrant.

As discussed above, $(A^n_{E_\mathrm{max}}(\omega) \tilde b(\omega))_{n \in \Z}$ corresponds to a solution $u_\omega$ of \eqref{e.ddeve} with $E = E_\mathrm{max}$, which must then also be bounded. Similarly, since $(A_*^n v)_{n \in \Z}$ is linearly growing in both directions (i.e., for $n \to \infty$ and for $n \to - \infty$) for any $v$ that is linearly independent from $(1,0)^t$, we see that all solutions of this difference equation that are linearly independent from $u_\omega$ must be linearly growing in both directions.

It follows that the solution $u_{\omega}$ is subordinate at both $+ \infty$ and $- \infty$. By Theorem~\ref{t.positivegroundstate} it is therefore strictly positive up to a multiplicative constant. This means that for a suitable $\mathrm{const} \not= 0$, $b := \mathrm{const} \cdot \tilde b$ takes values in the open first quadrant. By compactness of $\Omega$ and continuity of $b$, we can find a compact cone $C$ in the open first quadrant such that $b : \Omega \to C \setminus \{ 0 \}$, completing the proof.
\end{proof}

\end{appendix}

\end{document}